\documentclass[12pt]{amsart}
\usepackage[colorlinks=true,urlcolor=blue, citecolor=red,linkcolor=blue,linktocpage,pdfpagelabels, bookmarksnumbered,bookmarksopen]{hyperref}
\usepackage[hyperpageref]{backref}
\usepackage{amsthm} 
\usepackage{latexsym,amsmath,amssymb}
\usepackage{accents}
\usepackage[colorinlistoftodos,prependcaption,textsize=tiny]{todonotes}
\usepackage{a4wide}
\usepackage{soul}
\usepackage{mathtools} 
\usepackage{xparse} 
  
\title[Limits of Conformal Immersions]{Limits of conformal immersions under a bound on a fractional normal curvature quantity}

\author{Armin Schikorra}
\address[Armin Schikorra]{Department of Mathematics,
University of Pittsburgh,
301 Thackeray Hall,
Pittsburgh, PA 15260, USA}
\email{armin@pitt.edu}

plus 6pt minus 12pt
\def\eps{\varepsilon}

\def\B{{\mathbb{B}}}

\def\N{{\mathbb N}}

\def\S{{\mathbb S}}

\newtheorem{theorem}{Theorem}
\newtheorem{lemma}[theorem]{Lemma}

\newtheorem{proposition}[theorem]{Proposition}

\newtheorem{remark}[theorem]{Remark}

\def\supp{{\rm supp\,}}

\newcommand{\R}{\mathbb{R}}

\newcommand{\Z}{\mathbb{Z}}

\newcommand{\brac}[1]{\left (#1 \right )}

\newcommand{\Sw}{\mathcal{S}}

\newcommand{\barint}{
\rule[.036in]{.12in}{.009in}\kern-.16in \displaystyle\int }

\newcommand{\barcal}{\mbox{$ \rule[.036in]{.11in}{.007in}\kern-.128in\int $}}

\def\mvint_#1{\mathchoice
          {\mathop{\vrule width 6pt height 3 pt depth -2.5pt
                  \kern -8pt \intop}\nolimits_{\kern -3pt #1}}%
          {\mathop{\vrule width 5pt height 3 pt depth -2.6pt
                  \kern -6pt \intop}\nolimits_{#1}}%
          {\mathop{\vrule width 5pt height 3 pt depth -2.6pt
                  \kern -6pt \intop}\nolimits_{#1}}%
          {\mathop{\vrule width 5pt height 3 pt depth -2.6pt
                  \kern -6pt \intop}\nolimits_{#1}}}

\numberwithin{theorem}{section} \numberwithin{equation}{section}

\newcommand{\lap}{\Delta }
\newcommand{\aleq}{\precsim}
\newcommand{\ageq}{\succsim}
\newcommand{\aeq}{\approx}
\newcommand{\Rz}{\mathcal{R}}
\newcommand{\laps}[1]{(-\lap)^{\frac{#1}{2}}}
\newcommand{\lapv}{(-\lap)^{\frac{1}{4}}}
\newcommand{\lapms}[1]{I^{#1}}

\renewcommand{\div}{\operatorname{div}}

\let\latexchi\chi
\makeatletter
\renewcommand\chi{\@ifnextchar_\sub@chi\latexchi}
\newcommand{\sub@chi}[2]{
  \@ifnextchar^{\subsup@chi{#2}}{\latexchi^{}_{#2}}%
}
\newcommand{\subsup@chi}[3]{
  \latexchi_{#1}^{#3}%
}
\makeatother

\newcommand{\dv}{\operatorname{div}}
\begin{document}

\begin{abstract}
We consider limits of weakly converging $W^{1,2}$-maps $\Phi_k$ from a ball $B \subset \R^2$ into $\R^3$ which are conformal immersions. Under the assumption that a normal curvature term is small, namely if for the normal map $u$ we have for some $s \in (\frac{1}{2},1)$
\[
\int\limits_{B} \int\limits_{B} \left |
\frac{u_k(x) \wedge u_k(y)}{|x-y|^{s}} 
\right|^{\frac{2}{s}}\, \frac{dx\, dy}{|x-y|^{2}} < \eps
\] 
then we show that we can either pass to the limit and obtain an almost everywhere immersion $\Phi$ or $\Phi$ collapses and is constant. This is in the spirit of the results by T. Toro, and S. M\"uller and V.
Sverak, and F. H\'elein, who obtained similar statements under the stronger assumptions that the second fundamental form is bounded (but also stronger result: a locally bi-Lipschitz parametrization).

The fractional normal curvature assumption is vaguely reminiscent of curvature energies such as the scaling-invariant limits of tangent-point energies for surfaces as considered by Strzelecki, von der Mosel et al. and we hope that eventually the analysis in this work can be used to define weak immersions with these kind of energy bounds.
\end{abstract}

\maketitle
\tableofcontents
\sloppy

\section{Introduction}
Let $\Phi: B(0,1) \to \R^3$ be a conformal parametrization of a patch of a surface $\Sigma \subset \Phi(B(0,1))$. The Willmore energy of this patch is given as
\[
\mathcal{W}(\Phi;B(0,1)) = \frac{1}{4} \int_{B(0,1)} |\nabla u|^2 +C(\Sigma).
\]
where $u$ is the unit normal to $\Sigma$ at $\Phi(x)$.
\[
u(x) = \frac{\partial_1 \Phi }{|\partial_1 \Phi|}\wedge \frac{ \partial_2 \Phi}{ |\partial_2 \Phi|}.
\]
The following is a fundamental theorem by M\"uller-Sverak \cite{MS95} after earlier works by Toro \cite{T94,T95}, see also \cite[Theorem 5.1.1]{Helein-2002}. Sharp constants $\eps_0$ were obtained in \cite{KL12}, \cite{LiLuoTang13}. We also refer to surveys \cite{KS12} and \cite{R15}.  

\begin{theorem}\label{th:ms}
Assume that $\Phi_k \in C^\infty(\overline{B(0,1)},\R^3)$ is a sequence of conformal immersions, i.e. for each $k \in \N$
\[
\partial_\alpha \Phi_k = e^{\lambda_k} e_{\alpha;k} \quad \alpha = 1,2
\]
for some orthonormal system $e_{1;k},e_{2;k} \in C^\infty(\overline{B(0,1)},\R^3)$.

If $\Phi_k$ converges weakly to $\Phi$ in $W^{1,2}(B(0,1),\R^3)$ and if
\[
\sup_{k \in \N} \mathcal{W}(\Phi_k;B(0,1)) < \eps_0,
\]
then $\Phi$ is either a constant map or $\Phi$ is a bilipschitz conformal immersion. 

More precisely, there are $\lambda \in L^\infty \cap W^{1,2}_{loc}(B(0,1))$ and $(e_1,e_2)\in W^{s,\frac{2}{s}}_{loc}(B(0,1),\S^2)$ an orthonormal system such that
\[
\partial_\alpha \Phi = e^{\lambda} e_{\alpha} \quad \alpha = 1,2.
\]
\end{theorem}
This theorem has had numerous applications, in particular in the theory of weak Willmore surfaces, to define and analyze weak Willmore immersions, see the celebrated \cite{R08}.

On the other hand, in recent works \cite{BRS16,BRS19} there has been a breakthrough in the regularity theory of critical points of knot energies, namely M\"obius- and more generally scaling-invariant O'Hara energies, which we denote here by $\mathcal{O}^{\alpha,\frac{4}{\alpha}}$. These are self-repulsive energies on one-dimensional closed curves $\gamma: \S^1 \to \R^3$. The main idea in \cite{BRS16,BRS19} is that one can reduce the regularity analysis to the regularity analysis of fractional harmonic maps. Namely, by choosing the right parametrization (for curves: the constant-speed parametrization) we show that so parametrized critical knots $\gamma : \S^1 \to \R^3$ of the energy
\[
\mathcal{O}^{\alpha,\frac{4}{\alpha}}(\gamma)
\]
induce a critical map for 
\begin{equation}\label{eq:reformohara}
\mathcal{E}^{\alpha,\frac{4}{\alpha}}(v): \quad \mbox{subject to } v: \S^1 \to \S^2.
\end{equation}
Namely, if $\gamma: \S^1 \to \R^3$ is critical with respect to the O'Hara energy, then $v := \frac{\gamma'}{|\gamma'|}: \S^1 \to \S^2$ is a harmonic map with respect to the variational problem \eqref{eq:reformohara}. Moreover the energy $\mathcal{E}^{\alpha,\frac{4}{\alpha}}$ is essentially comparable to the $W^{\alpha,\frac{1}{\alpha}}$-Sobolev norm, which makes the regularity theory for critical points of \eqref{eq:reformohara} attainable from the regularity of degenerate harmonic maps, \cite{S15} -- a theory initiated for 1/2-harmonic maps by Da Lio and Rivi\`{e}re \cite{DLR11a,DLR11b}.

There are higher-dimensional analogues of surface energies, e.g. tangent-point and Menger-curvature energies have been extended \cite{StrzeleckivdM-JGA2013,KolasinskiStrzleeckivdMGAFA,StrzeleckivdMAdvMath2011}, and also analogues of the O'Hara energies exist \cite{OS18}. Nothing is known about the critical or even minimizing surfaces of these energies in the scale-invariant case.

In order to even have a glimpse of a chance of generalizing the arguments in \cite{BRS16,BRS19} to surfaces, the choice of parametrization (which in one-dimension is elementary) becomes a first roadblock.

Since in the case of Willmore surfaces the conformal parametrization has been the way to go, one might ask if some sort of conformal parametrization might also be possible for the fractional surface energies, i.e. if a version of Theorem~\ref{th:ms} holds in this case. This is by no means trivial, since the condition $\mathcal{W}(\Phi) < \eps_0$ is a condition of Sobolev-differential order $W^{2,2}$ of the surface. A brief and superficial analysis of the energies in \cite{StrzeleckivdM-JGA2013,KolasinskiStrzleeckivdMGAFA,StrzeleckivdMAdvMath2011} exhibits however that these are of order $W^{s,\frac{2}{s}}$ of the surface, where $s \in (1/2,1)$. As of now we are not able to replace $\mathcal{W}(\Phi,B(0,1))$ in Theorem~\ref{th:ms} by a condition on the tangent-point or Menger curvature of surfaces. However, our main result in Theorem~\ref{th:fracms} below is that we can lower the differential order for an energy that at least formally is reminiscent of these energies. 

The $W^{1,2}$-case is based on the second fundamental form $|A| = |\nabla u| = |u \wedge \nabla u|$ and the smallness condition in Theorem~\ref{th:ms} is with respect to the Willmore energy,
\[
\mathcal{W}_{1,2}(\Phi,B(0,1)) = \|A\|_{L^2(B(0,1)}^2 = \|u \wedge \nabla u\|_{L^2(B(0,1))}^2.
\]
Instead, for $s \in (\frac{1}{2},1)$ we consider
\[
\mathcal{W}_{s,p}(\Phi,B(0,1)) := \int_{B(0,1)} \int_{B(0,1)} \frac{|u(x) \wedge u(y)|^p}{|x-y|^{n+sp}}\, dx\, dy
\]
The notion $\mathcal{W}_{s,p}$ is somewhat justified, since by the same argument as, e.g., in \cite{BBM01} we have
\begin{lemma}
For $s \in (0,1)$, we have
\[
\mathcal{W}_{s,\frac{2}{s}}(\Phi,B(0,1)) \aleq \mathcal{W}_{1,2}(\Phi,B(0,1)).
\]
Moreover,
\[
\lim_{s \to 1} (1-s) \mathcal{W}_{s,2}(\Phi,B(0,1)) = \mathcal{W}_{1,2}(\Phi,B(0,1)),
\]
\end{lemma}
Our main result is the following extension of Theorem~\ref{th:ms}.
\begin{theorem}\label{th:fracms}
For any $s \in (\frac{1}{2},1]$ there exists $\eps_s$ such that the following holds.
For some ball $B \subset \R^2$ assume that $\Phi_k \in C^\infty(\overline{B},\R^3)$ is a sequence of conformal immersions, i.e. for each $k \in \N$
\[
\partial_\alpha \Phi_k = e^{\lambda_k} e_{\alpha;k} \quad \alpha = 1,2
\]
for some orthonormal system $e_{1;k},e_{2;k} \in C^\infty(\overline{B},\R^3)$.

If $\Phi_k$ converges weakly to $\Phi$ in $W^{1,2}(B,\R^3)$ and if
\[
\sup_{k \in \N} \mathcal{W}_{s,\frac{2}{s}}(\Phi_k;B) < \eps_s
\]
then $\Phi$ is either a constant map or $\Phi$ is a conformal immersion almost everywhere. More precisely there are $\lambda \in W^{1,2}_{loc}(B)$ and $(e_1,e_2)\in W^{s,\frac{2}{s}}_{loc}(B,\S^2)$ an orthonormal system such that
\[
\partial_\alpha \Phi = e^{\lambda} e_{\alpha} \quad \alpha = 1,2 \quad \mbox{a.e. in $B$}
\]
\end{theorem}

While Theorem~\ref{th:fracms} seems to be the first result in this direction for condition on the surface of differential order below $2$, there is one major drawback here: the quantity $\mathcal{W}_{s,\frac{2}{s}}$ is not very geometric, and it is not clear to the author which reasonable parametrization invariant surface energy reduces to $\mathcal{W}_{s,\frac{2}{s}}$ under the assumption of conformal parametrization.

We also have a few technical limitations (which can be remedied however):
\begin{itemize} 
\item The case $s=\frac{1}{2}$ is ruled out in Theorem~\ref{th:fracms}. But one observes that the analysis presented below can be extended to a version for $s=\frac{1}{2}$, however the formulation is quite technical: If one assumes that 
\[
\sup_{k \in \N} \|\lapv u_k\|_{L^{4}(B(0,1))} < \eps_{\frac{1}{2}}
\]
for maps $u_k: \R^2 \to \R^{3}$ which on $B(0,1)$ coincide with the unit normal $u_k$ of $\Phi(B(0,1))$. We did not find a representation of this fact in terms of a functional more reminiscent of $\mathcal{W}_{s,\frac{2}{s}}$.

\item One also notices that we do not obtain an analogue of the $L^\infty_{loc}$-control of the conformal factor. From the point of view of harmonic analysis it seems unlikely that such a uniform boundedness holds under the assumptions presented. 
It is however a relatively easy consequence of our argument that if in addition for any extension $u_k$ of the unit normal $u_k: B(0,1) \to \S^{2}$ we have
\begin{equation}\label{eq:ulorentzbd}
\sup_{k \in \N} \|\laps{s} u_k\|_{L^{(\frac{n}{s},2)}(B(0,1))} < \infty
\end{equation}
then we have the a local bound on the conformal factor. Here $L^{(\frac{n}{s},2)}$ denotes the Lorentz space, and by Sobolev inequality it is still true that $\mathcal{W}_{1,2}(\Phi_k;B(0,1)) < \infty$ then \eqref{eq:ulorentzbd} holds. In this sense it is possible to obtain a real extension of Theorem~\ref{th:ms} -- however the geometric meaning is unclear.
\end{itemize}
Let us remark on the proof of Theorem~\ref{th:fracms}. In the classical case $s=1$, the proof of Theorem~\ref{th:ms} can be based on the following lifting property, see \cite[Lemma 5.1.5]{Helein-2002}.

\begin{theorem}\label{th:lifting}
There exists $\eps_1 > 0$ so that the following holds true.

Let $B \subset \R^2$ be any ball. Assume that $(u,e_1,e_2) \in C^\infty(\overline{B},\S^2)$ form pointwise an orthonormal basis of $\R^3$.

If $u$ satisfies
\begin{equation}\label{eq:ucondold}
 \int_{B} |\nabla u|^2 \leq \eps_1
\end{equation}
then there exist $\tilde{e}_1$ and $\tilde{e}_2 \in \dot{H}^{1,2}(\R^2,\S^2)$ such that
\begin{equation}
 (\tilde{e}_1, \tilde{e}_2, u ) \quad \mbox{form an orthonormal basis of $\R^3$ a.e.},
\end{equation}
such that 
\[
 \dv(\langle \tilde{e}_1, \nabla \tilde{e}_2 \rangle_{\R^3}) = 0,
\]
and such that
\begin{equation}
 \|\langle \tilde{e}_1,\nabla \tilde{e}_2 \rangle \|_{L^2(B)} \aleq \|\nabla u\|_{L^2(B)}^2.
\end{equation}
\end{theorem}

The proof of Theorem~\ref{th:fracms} is based on the fact that one can sharpen Theorem~\ref{th:lifting}. 

\begin{theorem}\label{th:frlifting}
For any $s \in (\frac{1}{2},1]$ there exists $\eps_s > 0$ so that the following holds true.

Let $B \subset \R^2$ be any ball. Assume that $(u,e_1,e_2) \in C^\infty(\overline{B},\S^2)$ form pointwise an orthonormal basis of $\R^3$.

If 
\begin{equation}\label{eq:uwucond}
\int_{B} \int_{B} \frac{|u(x) \wedge u(y)|^{\frac{2}{s}}}{|x-y|^{4}}\, dx\, dy \leq \eps_s
\end{equation}
then there exist $\tilde{e}_1$ and $\tilde{e}_2 \in W^{s,\frac{n}{s}}(B(0,1),\S^2)$ such that
\[
 (\tilde{e}_1, \tilde{e}_2, u ) \quad \mbox{form an orthonormal basis of $\R^3$ a.e. in $B(0,1)$},
\]
and
\[
 \dv(\langle \tilde{e}_1, \nabla \tilde{e}_2 \rangle_{\R^3}) = 0 \quad \mbox{in $B$},
\]
and 
\begin{equation}\label{eq:frliftest}
[\tilde{e}_1]_{W^{s,\frac{2}{s}}(B)}^{\frac{2}{s}} + [\tilde{e}_2]_{W^{s,\frac{2}{s}}(B)}^{\frac{2}{s}}  + \|\langle \tilde{e}_1,\nabla \tilde{e}_2 \rangle \|_{L^2}  \aleq \int_{B} \int_{B} \frac{|u(x) \wedge u(y)|^{\frac{2}{s}}}{|x-y|^{4}}\, dx\, dy.
\end{equation}
\end{theorem}
\begin{remark}
At least for $s < 1$ large enough, one can also prove a version of Theorem~\ref{th:frlifting} in the setting of Lorentz spaces, replacing the smallness assumption \eqref{eq:uwucond} with the condition that $\|\laps{s} u\|_{L^{(2,\infty)}} < \eps_s$. For $s =1$ this was used for the energy quantization analysis of the Willmore energy in neck regions, see \cite[Lemma IV.3]{BR14} and \cite{LR18}. We will not follow this path here.
\end{remark}
The outline of the remaining paper is as follows. In Section~\ref{s:ingredients} we state the main technical results that lead to a proof of Theorem~\ref{th:frlifting} and \ref{th:fracms}. The proofs of Theorem~\ref{th:frlifting} is then given in Section~\ref{s:prooffrlifting} and the proof of Theorem~\ref{th:fracms} is given in Section~\ref{s:prooffracms}. The results of Section~\ref{s:ingredients} are proven in the remaining chapters.

\section{Main technical ingredients}\label{s:ingredients}
In this section we state the main technical ingredients needed for the proof of Theorem~\ref{th:fracms} and \ref{th:frlifting}. These are mainly sharpening of classical local results to the fractional case (which is surprisingly involved for some of these results). We first introduce the fractional Sobolev space. For a more thorough (and technical) introduction we refer to Section~\ref{s:fracsob}.

Let $\Omega \subset \R^n$ be an open set. For $s \in (0,1)$ the Slobodeckij-Gagliardo Sobolev space $W^{s,p}(\Omega)$ is defined as all functions $f \in L^p(\Omega)$ such that $[f]_{W^{s,p}(\Omega)} < \infty$. Here, we use the Gagliardo seminorm
\[
[f]_{W^{s,p}(\Omega)} := \brac{\int_{\Omega} \int_{\Omega} \frac{|f(x)-f(y)|^p}{|x-y|^{n+sp}}\, dx\, dy }^{\frac{1}{p}}.
\]
For $s = 1$ we use the (abuse of, see below) notation
\[
[f]_{W^{1,p}(\Omega)} = \|\nabla f\|_{L^p(\Omega)}.
\]
For $s = 0$ we denote $[f]_{W^{0,p}(\Omega)} = \|f\|_{L^p(\Omega)}$.

Now let us begin with our first ingredient. If $u: B \to \S^2$ then $u \cdot \nabla u = 0$. By Lagrange's identity we then get
\[
\|\nabla u\|_{L^2(B)} \aeq \|u \wedge \nabla u\|_{L^2(B)}.
\]
In some sense this still holds for our fractional normal curvature quantity. One way is easy,  
Since $u(x) \wedge u(x) = 0$ and thus $|u(x) \wedge u(y)| \aleq |u(x)|\, |u(x)-u(y)|$ we have the trivial estimate 
\[
\int_{B}\int_{B}\frac{|u(x)\wedge u(y)|^{\frac{2}{s}}}{|x-y|^{4}}\, dx\, dy \leq 
\|u\|_{L^\infty(B)}^{\frac{2}{s}}\ [u]_{W^{s,\frac{2}{s}}(B)}^{\frac{2}{s}}.
\]
The other direction in general has no reason to hold, however it holds when the fractional normal curvature quantity is small.
\begin{proposition}\label{pr:uwusmalluwsmall}
For $s \in (0,1)$ there exists $\eps_s > 0$ such that the following holds. For any ball $B \subset \R^2$ and any $u \in W^{s,\frac{2}{s}}(B,\S^{2})$, if
\[
\int_{B}\int_{B}\frac{|u(x)\wedge u(y)|^{\frac{2}{s}}}{|x-y|^{4}}\, dx\, dy < \eps_s
\]
then
\[
[u]_{W^{s,\frac{2}{s}}(B)} \aeq \brac{\int_{B}\int_{B}\frac{|u(x)\wedge u(y)|^{\frac{2}{s}}}{|x-y|^{4}}\, dx\, dy}^{\frac{s}{2}}.
\]
\end{proposition}
Proposition~\ref{pr:uwusmalluwsmall} will be proven in Section~\ref{s:proofuwusmall}.

Another ingredient is an extension to the fractional order of the following (again almost trivial) observation:

If $(e_1,e_2,u)$ for almost everywhere a orthonormal basis of $\R^3$ then, observing that $e_1 \cdot \nabla e_1 \equiv 0$,
\[
\begin{split}
|\nabla e_1| \leq& |\langle u , \nabla e_1\rangle | + |\langle e_2, \nabla e_1 \rangle| \\
=& |\langle \nabla u , e_1\rangle | + |\langle e_2, \nabla e_1 \rangle| \\
\leq& |\nabla u| + |\langle e_2, \nabla e_1 \rangle|.
\end{split}
\]
In particular the $W^{1,2}$-energy of $e_1$ is controlled by $\langle e_1,\nabla e_2\rangle$ and $\nabla u$. This latter fact still holds somewhat true in the fractional case but the proof is much more involved and relies on several commutator estimates. 
\begin{proposition}\label{pr:eest}
Let $s \in (\frac{1}{2},1]$ and and $B \subset \R^2$ be a ball or all of $\R^2$.

For any orthonormal system $(e_1,e_2,u)$ of $\R^3$ we have the estimate
\[
[e_1]_{W^{s,\frac{2}{s}}(B)} + [e_2]_{W^{s,\frac{2}{s}}(B)} \aleq \|\langle e_1, \nabla e_2 \rangle \|_{L^2(B)} + [u]_{W^{s,\frac{2}{s}}(B)} + [u]_{W^{s,\frac{2}{s}}(B)}\, \brac{[e_1]_{W^{s,\frac{2}{s}}(B)} + [e_2]_{W^{s,\frac{2}{s}}(B)}}.
\]
In particular there exists $\eps_s > 0$ such that if for any orthonormal system $(e_1,e_2,u)$ of $\R^3$ we have
\[
[u]_{W^{s,\frac{2}{s}}(B)} \leq \eps_s,
\]
then
\[
[e_1]_{W^{s,\frac{2}{s}}(B)} + [e_2]_{W^{s,\frac{2}{s}}(B)} \aleq \|\langle e_1 \nabla e_2 \rangle \|_{L^2(B)} + [u]_{W^{s,\frac{2}{s}}(B)}.
\]
\end{proposition}
Proposition~\ref{pr:eest} will be proven in Section~\ref{s:eest}.

We will also employ a refined version of Wente's inequality
\begin{proposition}\label{pr:wente1}
Let $B \subset \R^2$ be a ball. Assume $\lambda_0 \in H^1_0(B)$ satisfies an equation
\[
\begin{cases}
\lap \lambda_0 = \langle \nabla^\perp f,\nabla g \rangle \quad& \mbox{in $B$}\\
\lambda_0 = 0 \quad &\mbox{on $\partial B$}.
\end{cases}
\]
Then for any $s \in (\frac{1}{2},1]$,
\[
\|\nabla \lambda_0\|_{L^2(B)} \leq C\, [f]_{W^{s,\frac{2}{s}}(B)}\, [g]_{W^{s,\frac{2}{s}}(B)}.
\]
\end{proposition}
Wente's inequality leads also to the following estimate.
\begin{proposition}\label{pr:wenteour}
For any $s \in (\frac{1}{2},1]$ there exists $\eps_s > 0$ such that the following holds. 

Let $B \subset \R^2$ be a ball. For $e_1,e_2,u \in H^1(B,\S^2)$ that form a.e. orthonormal basis of $\R^3$ and that allow for the existence of $\lambda_0 \in H^1_0(B)$ satisfying
\[
\begin{cases}
\lap \lambda_0 = \langle \nabla^\perp e_1,\nabla e_2 \rangle \quad& \mbox{in $B$}\\
\lambda_0 = 0 \quad &\mbox{on $\partial B$}
\end{cases}
\]
If
\[
[u]_{W^{s,\frac{n}{s}}(B)} < \eps_s
\]
then we have have
\[
\|\nabla \lambda_0\|_{L^2(B)} \leq C_s\, \brac{\|\langle e_1, \nabla e_2 \rangle\|_{L^2(B)}^2 + [u]_{W^{s,\frac{2}{s}}(B)}^2}.
\]
The constant $C_s$ in the estimate is independent of the particular ball.
\end{proposition}
Proposition~\ref{pr:wente1} and Proposition~\ref{pr:wenteour} will be proven in Section~\ref{s:proofwenteour}.

\section{Proof of the lifting theorem, Theorem~\ref{th:frlifting}}\label{s:prooffrlifting}

The proof of Theorem~\ref{th:frlifting} is an adaptation of the proof in \cite[Lemma 5.1.4]{Helein-2002}, which is based on a continuity argument. For this we first observe the following continuous dependence for a variational problem.

\begin{lemma}[Continuous dependence of Gauge]\label{la:contgauge}
For $g \in L^2(\R^2,\R^2)$ let
\[
f(r) :=  \inf_{\theta} \|\nabla \theta + g\|_{L^2(B(0,r)},
\]
where the infimum is taken over $\theta \in W^{1,2}(B(0,r))$.

Then $f$ is continuous in $(0,\infty)$. 
Moreover, $\lim_{r \to 0^+} f(r) = 0$.
\end{lemma}
\begin{proof}
For the limit at $0$ we have from the minimization property
\[
\limsup_{r \to 0^+} f(r) \leq \limsup_{r \to 0} \|g\|_{2,B(0,r)} = 0,
\]
by absolute continuity of the integral.

Regarding the continuity we first observe that $f$ is monotone. Indeed let $0 < r < R$, and let $\theta_R$ be a minimizer (which always exists) for $f(R)$. Then $\theta_R$ is a competitor for $f(r)$ and thus
\[
f(r) \leq \|\nabla \theta_R + g\|_{L^2(B(0,r)} \leq \|\nabla \theta_R + g\|_{L^2(B(0,R))}=f(R).
\]
Next, set
\[
\tilde{g}(x) := \frac{r}{R} g\left (\frac{r}{R}x \right ).
\]
Denote by $\tilde{\theta}_R$ a minimizer of
\[
\tilde{f}(R) := \|\nabla \tilde{\theta} + \tilde{g}\|_{L^2(B(0,R))}.
\]
We then have
\[
f(r) = \tilde{f}(R).
\]
From this we conclude
\[
\begin{split}
0 \leq& \brac{f(R)}^2 - \brac{f(r)}^2 = \brac{f(R)}^2 - \brac{\tilde{f}(R)}^2\\
\leq&\int_{B(0,R)} |\nabla \tilde{\theta}_R + g|^2-|\nabla \tilde{\theta}_R + \tilde{g}|^2\\
=&\int_{B(0,R)} |g|^2-|\tilde{g}|^2 + 2\nabla \tilde{\theta}_R \cdot (g -\tilde{g})\\
\aleq& \brac{\|g\|_{L^2(B(0,R))} + \|\tilde{g}\|_{L^2(B(0,R))}+ \|\nabla \tilde{\theta}_R\|_{L^2(B(0,R))}}\, \|g-\tilde{g}\|_{L^2(B(0,R))}
\end{split}
\]
Since $\Theta_R$ is a minimizer of $\tilde{f}(R)$ this implies in particular,
\[
0 \leq \brac{f(R)}^2 - \brac{f(r)}^2 \aleq \brac{\|g\|_{L^2(B(0,R))} + \|\tilde{g}\|_{L^2(B(0,R))}}\, \|g-\tilde{g}\|_{L^2(B(0,R))}
\]
Since we have by dominated convergence,
\[
\lim_{r \to R} \|g-\tilde{g}\|_{L^2(B(0,R))} = \lim_{R \to r}\|g-\tilde{g}\|_{L^2(B(0,R))} = 0,
\]
we obtain continuity for $f$.
\end{proof}

\begin{proof}[Proof of Theorem~\ref{th:frlifting}]

Since the situation is scaling invariant, we assume w.l.o.g. $B = B(0,1)$. Moreover, by Proposition~\ref{pr:uwusmalluwsmall}, since $s > \frac{1}{2}$, we may assume w.l.o.g. that instead of \eqref{eq:uwucond} we have
\begin{equation}\label{eq:uwucond2}
[u]_{W^{s,\frac{n}{s}}(B(0,1))} \leq \eps_s.
\end{equation}

For $r \in (0,1]$ we consider the minimization problem
\[
\mathcal{E}_r(\theta) := \int_{B(0,r)} |\langle \tilde{e}_1(\theta), \nabla \tilde{e}_2(\theta) \rangle |^2
\]
where
\[
\tilde{e}_\alpha(\theta) := P_{\alpha \beta}(\theta)\, e_\beta,
\]
and $P \in W^{1,2}(B(0,r)),SO(N))$ is a rotation parametrized by $\theta$, namely
\[
 P(\theta) = \left ( \begin{array}{cc}
                             \cos(\theta) & -\sin(\theta)\\
                             \sin(\theta) & \cos(\theta)
                            \end{array} \right ).
\]
Equivalently we can then write
\[
\mathcal{E}_r(\theta) := \int_{B(0,r)} |\nabla \theta + \langle e_1, \nabla e_2 \rangle|^2.
\]
See the proof of Lemma~\ref{la:thetaconfparam} for the relevant computations.

We minimize $\mathcal{E}_r(\theta)$ in $\theta \in W^{1,2}(B(0,r))$ and $\int_{B(0,r)} \theta = 0$. We call the minimizing frame $\tilde{e}_{\alpha,r} := \tilde{e}_{\alpha}(\theta_r)$. It satisfies the equation
\[
\begin{cases}
\div(\langle e_{1,r}, \nabla e_{2,r} \rangle) = 0 &\quad \mbox{in $B(0,r)$},\\
\langle e_{1,r}, \partial_\nu e_{2,r} \rangle = 0 &\quad \mbox{on $\partial B(0,r)$}.
\end{cases}
\]
By Hodge decomposition we thus find $\lambda_r \in W^{1,2}_0(B(0,r))$ such that 
\[
\begin{cases}
\nabla^\perp \lambda_r = \langle e_{1,r}, \nabla e_{2,r} \rangle &\quad \mbox{in $B(0,r)$},\\
\lambda_r = 0 &\quad \mbox{on $\partial B(0,r)$}.
\end{cases}
\]
Taking the curl on both sides of the equation we find the equation in Proposition~\ref{pr:wenteour}.
\[
\begin{cases}
\lap \lambda_r = \langle \nabla^\perp e_{1,r}, \nabla e_{2,r} \rangle &\quad \mbox{in $B(0,r)$},\\
\lambda_r = 0 &\quad \mbox{on $\partial B(0,r)$}.
\end{cases}
\]
From Proposition~\ref{pr:wenteour} we then find (for suitably small $\eps_s$),
\[
\begin{split}
\|\nabla \lambda_r\|_{L^2(B(0,r))} \leq& C_s\, \brac{\|\langle e_1,\nabla e_{2_r} \rangle\|_{L^2(B(0,r))}^2 + [u]_{W^{s,\frac{2}{s}}(B(0,1))}^2}\\
=& C_s\, \brac{\|\nabla^\perp \lambda_{r}\|_{L^2(B(0,r))}^2 + [u]_{W^{s,\frac{2}{s}}(B(0,1))}^2}.
\end{split}
\]
If we set
\[
f(r) := \|\nabla \lambda_r\|_{L^2(B(0,r))} = \inf_{\theta} \mathcal{E}_r(\theta)
\]
we arrive at
\begin{equation}\label{eq:frpol}
f(r) \leq C \brac{(f(r))^2 + [u]_{W^{s,\frac{2}{s}}(B(0,1))}^2}.
\end{equation}
Observe that for $\eps_s$ small enough the roots of the polynomial \[C x^2 + [u]_{W^{s,\frac{2}{s}}(B(0,1))}^2 - x\] are real numbers,
\[
F_1(r) := \frac{1}{2C} - \sqrt{\frac{1}{4C^2} - \frac{[u]_{W^{s,\frac{2}{s}}(B(0,1))}^2}{C}}, \quad F_2(r) := \frac{1}{2C} + \sqrt{\frac{1}{4C^2} - \frac{[u]_{W^{s,\frac{2}{s}}(B(0,1))}^2}{C}}.
\]
That is, \eqref{eq:frpol} implies
\begin{equation}\label{eq:frp}
\mbox{either}\ f(r) \in [0,F_1(r)] \quad \mbox{or} \quad f(r) \in [F_2(r),\infty) 
\end{equation}
In view of Lemma~\ref{la:contgauge} $f(r)$ is continuous and $\lim_{r \to 0} f(r) = 0$, which implies
\[
f(r) \in [0,F_1(r)] \quad \mbox{for all $r \in (0,1]$}.
\]
In particular, we obtain
\[
\|\nabla \lambda_r\|_{L^2(B(0,r))} =f(r) \aleq [u]_{W^{s,\frac{2}{s}}(B(0,1))}^2 \quad \mbox{for all $r \in (0,1]$}.
\]
Setting $\tilde{e}_\alpha = \tilde{e}_{\alpha,1}$ we have obtained the estimate
\[
 \|\langle \tilde{e}_1,\nabla \tilde{e}_2 \rangle \|_{L^2(B(0,1))}  \aleq [u]_{W^{s,\frac{2}{s}}(B(0,1))}^2.
\]
The estimate \eqref{eq:frliftest} is then a consequence of Proposition~\ref{pr:eest}. 
\end{proof}

\section{Limits of conformal maps, Proof of Theorem~\ref{th:fracms}}\label{s:prooffracms}
hanThe proof of Theorem~\ref{th:fracms} follows by an adaptation of the argument in \cite[Theorem 5.1.1]{Helein-2002}. We mainly need to keep track of the improved estimates from Section~\ref{s:ingredients}.

We begin with some standard observations about conformal maps.
\begin{lemma}\label{la:conflambda1}
Let $\Phi \in C^\infty(B(0,1),\R^3)$ be a conformal map. That is assume that for some orthonormal basis $e_1,e_2$ of the tangent space of $\Phi(B(0,1))$ and some $\lambda \in C^\infty(B(0,1))$ we have
\begin{equation}\label{eq:confphi}
\partial_\alpha \Phi = e^\lambda e_\alpha, \quad \alpha = 1,2 \quad \mbox{in $B(0,1)$}
\end{equation}
Then
\begin{equation}\label{eq:lambdate1e2}
-\nabla^\perp \lambda = \langle e_1, \nabla e_2 \rangle_{\R^3}\\
\end{equation}
\end{lemma}
\begin{proof}
Taking the curl on both sides of \eqref{eq:confphi} we obtain
\[
\begin{split}
 0 =& \partial_2 (e^\lambda e_1) - \partial_1 (e^\lambda e_2).
\end{split}
 \]
Multiplying by $e^{-\lambda}$ we thus have
\begin{equation}\label{eq:0eq}
 0 = \partial_2 \lambda\, e_1 + \partial_2 e_1 - \partial_1 \lambda\, e_2- \partial_1 e_2 \\
\end{equation}
Observe that
\[
\langle e_\beta, e_\gamma \rangle = \delta_{\beta \gamma}
\]
readily implies
\[
\langle \partial_\alpha e_\beta, e_\gamma \rangle = 0 \quad \mbox{for $\beta = \gamma$}.
\]
Then the claim follows from taking the scalar product in \eqref{eq:0eq} with $e_1$ and $e_{2}$.
\end{proof}

When we change the basis of the tangent space from $(e_1,e_2)$ into a new basis $(\tilde{e}_1,\tilde{e}_2)$ with same orientation then \eqref{eq:lambdate1e2} changes accordingly. Namely we have

\begin{lemma}\label{la:thetaconfparam}
Let $\Phi  \in C^\infty(B(0,1),\R^3)$ be conformal, say as in Lemma~\ref{la:conflambda1}
\[
\partial_\alpha \Phi  = e^\lambda e_\alpha, \quad \alpha = 1,2.
\]
Let $\tilde{e}_1, \tilde{e}_2$ be any other orthonormal basis with the same orientation, i.e. assume that $\tilde{e}_\alpha = P_{\alpha \beta} e_\beta$ for some $P \in SO(2)$. If we represent 
\[
 P = \left ( \begin{array}{cc}
                             \cos(\theta) & -\sin(\theta)\\
                             \sin(\theta) & \cos(\theta)
                            \end{array} \right ).
\]
then we have
\begin{equation}\label{eq:lambdate1e2th}
- \nabla^\perp \lambda = \langle \tilde{e}_1,\nabla \tilde{e}_2 \rangle_{\R^3} + \nabla \theta
\end{equation}
\end{lemma}
\begin{proof}
Observe that 
\[
 \partial_\sigma \left ( \begin{array}{cc}
                             \cos(\theta) & -\sin(\theta)\\
                             \sin(\theta) & \cos(\theta)
                            \end{array} \right )= 
\left ( \begin{array}{cc}
                             -\sin(\theta) & -\cos(\theta)\\
                             \cos(\theta) & -\sin(\theta)
                            \end{array} \right ) \partial_\sigma \theta
\]
and in particular,
\[
 P \partial_\sigma P^T = \left ( \begin{array}{cc}
                             \cos(\theta) & -\sin(\theta)\\
                             \sin(\theta) & \cos(\theta)
                            \end{array} \right )
                            \partial_\sigma \left ( \begin{array}{cc}
                             \cos(\theta) & \sin(\theta)\\
                             -\sin(\theta) & \cos(\theta)
                            \end{array} \right ) = \left ( \begin{array}{cc}
                             0 & 1\\
                             -1 & 0
                            \end{array} \right )  \partial_\sigma \theta
\]
This implies 
\[
\begin{split}
 &\langle e_1, \nabla e_2 \rangle_{\R^3} \\
 =& P_{1\alpha} \nabla P_{2\beta} \langle \tilde{e}_\alpha,\tilde{e}_\beta\rangle_{\R^3}+ 
 P_{1\alpha} P_{2\beta}\langle \tilde{e}_\alpha , \nabla \tilde{e}_\beta\rangle_{\R^3}\\
 =& P_{1\alpha} \nabla P_{2\alpha} + 
 P_{11} P_{22}\langle \tilde{e}_1, \nabla \tilde{e}_2\rangle_{\R^3} + P_{12} P_{21}\langle \tilde{e}_2 , \nabla \tilde{e}_1\rangle_{\R^3}\\
 =& P_{1\alpha} \nabla P_{2\alpha} + 
 \brac{P_{11} P_{22} - P_{12} P_{21}}\langle \tilde{e}_1 , \nabla \tilde{e}_2\rangle_{\R^3}\\
  =& \nabla \theta + 
 \langle \tilde{e}_1,\nabla \tilde{e}_2 \rangle_{\R^3}\\
\end{split}
 \]
Now we obtain the claim from Lemma~\ref{la:conflambda1}, more precisely from \eqref{eq:lambdate1e2}.
\end{proof}

\begin{lemma}\label{la:lambdahestK}
Let $\Phi  \in C^\infty(\overline{B(0,1)},\R^3)$ be conformal, say
\begin{equation}\label{eq:phiconf1}
\partial_\alpha \Phi  = e^\lambda e_\alpha, \quad \alpha = 1,2
\end{equation}
for some orthonormal basis $e_1,e_2$ of the tangent space of $\Phi (B(0,1))$ and some $\lambda \in C^\infty(\overline{B(0,1)})$.

Let $\tilde{e}_1, \tilde{e}_2$ be any other orthonormal basis with the same orientation, i.e. $\tilde{e}_\alpha = P_{\alpha \beta} e_\beta$, for some $P \in C^\infty(\overline{B(0,1)},SO(2))$.
Let $\lambda^h$ be a solution to
\[
\begin{cases}
\lap \lambda^h = 0 \quad& \mbox{in $B(0,1)$},\\
\lambda^h = \lambda \quad &\mbox{on $\partial B(0,1)$}.
\end{cases}
\]
Then for every $r \in (0,1)$ and any $s \in (\frac{1}{2},1]$ there exist positive constants $C_1(r,s)$ and $C_2(r)$ so that
\[
\sup_{z \in B(0,r)} \lambda^h(z) \leq C_1(r,s) \brac{[\tilde{e}_1]^2_{W^{s,\frac{2}{s}}(B(0,1))}+ [\tilde{e}_2]^2_{W^{s,\frac{2}{s}}(B(0,1))} + \|\nabla \Phi \|_{L^2} } -C_2(r)\int_{B(0,1)} (\lambda^h)_- 
\]
where we use the notation $f_- = |\min \{f,0\}|$.
\end{lemma}
\begin{proof}
In what follows we shall use the notation
\[
\lambda^0 := \lambda-\lambda^h,
\]
which in view of Lemma~\ref{la:thetaconfparam} leads to the decomposition
\[
 - \nabla^\perp \lambda_0 = \langle \tilde{e}_1,\nabla \tilde{e}_2 \rangle_{\R^3} + \nabla \theta  + \nabla^\perp \lambda^h
\]
where $\theta$ is from the representation of $P$,
\[
 P = \left ( \begin{array}{cc}
                             \cos(\theta) & -\sin(\theta)\\
                             \sin(\theta) & \cos(\theta)
                            \end{array} \right ).
\]
Taking the curl on both sides (recall $\lap \lambda^h = 0$) we obtain in particular,
\begin{equation}\label{eq:lambda0eq}
\begin{cases} 
\lap \lambda_0 = \langle \nabla^\perp \tilde{e}_1, \nabla \tilde{e}_2 \rangle \quad &\mbox{in $B(0,1)$,}\\
\lambda_0 = 0 \quad& \mbox{on $\partial B(0,1)$}.
\end{cases}
\end{equation}
We now begin to estimate $\lambda^h$. Since $\lambda$ is harmonic, more precisely by Lemma~\ref{la:harmonic1}, we have
\[
\sup_{B(0,r)} \lambda^h \leq C_1(r) \int_{B(0,1)} \lambda^h_+ - C_2(r) \int_{B(0,1)} \lambda^h_-.
\]
Observe that from $\lambda^h = \lambda^h-\lambda + \lambda$ and $\lambda \leq e^\lambda$ we obtain the estimate
\[
(\lambda^h)_+ \leq |\lambda^h - \lambda| + e^\lambda
\]
and arrive at
\[
\sup_{B(0,r)} \lambda^h \leq C_1(r) \brac{\int_{B(0,1)} |\lambda_0| + \int_{B(0,1)} e^\lambda} - C_2(r) \int_{B(0,1)} \lambda^h_-
\]
Next we observe that from \eqref{eq:phiconf1} we have $|\Phi| = e^\lambda$ which leads to
\[
\int_{B(0,1)} e^{\lambda} \aleq \brac{\int_{B(0,1)} e^{2\lambda}}^{\frac{1}{2}} = \|\nabla \Phi \|_{L^2(B(0,1))}.
\]
That is, we arrive at
\[
\sup_{B(0,r)} \lambda^h \leq C_1(r) \brac{\int_{B(0,1)} |\lambda_0| + \|\nabla \Phi \|_{L^2(B(0,1)}} - C_2(r) \int_{B(0,1)} \lambda^h_-.
\]
By Poincar\`e-inequality, since $\lambda_0$ has trivial Dirichlet boundary data,
\[
\int_{B(0,1)} |\lambda_0| \aleq \|\nabla \lambda_0\|_{L^2(B(0,1))}.
\]
In view of \eqref{eq:lambda0eq}, an application of Wente's inequality in the form of Proposition~\ref{pr:wente1} leads to
\[
\|\nabla \lambda_0\|_{L^2(B(0,1))} \aleq [e_1]_{W^{s,\frac{2}{s}}(B(0,1))}\, [e_2]_{W^{s,\frac{2}{s}}(B(0,1))}.
\]
This establishes the claim.
\end{proof}

Combining the previous lemma, Lemma~\ref{la:lambdahestK}, with the lifting result, Theorem~\ref{th:frlifting}, we obtain

\begin{lemma}\label{la:nablaphiest}
For any $s \in (\frac{1}{2},1]$ there exists $\eps_s > 0$ such that the following holds.

Let $\Phi  \in C^\infty(B(0,1),\R^3)$ be conformal, say
\begin{equation}\label{eq:phiconformal2}
\partial_\alpha \Phi  = e^\lambda e_\alpha, \quad \alpha = 1,2
\end{equation}
holds for some orthonormal basis $e_1,e_2$ of the tangent space of $\Phi (B(0,1))$ and some $\lambda \in C^\infty(\overline{B(0,1)})$.

If for $u := e_1 \wedge e_2$ we have 
\[
\mathcal{W}_{s,\frac{2}{s}}(\Phi,B(0,1)) := \int_{B(0,1)} \int_{B(0,1)} \frac{|u(x) \wedge u(y)|^p}{|x-y|^{n+sp}}\, dx\, dy < \eps_s
\]
then for every $r \in (0,1)$ there exist constants $C_1(r,s), C_2(r,s) > 0$ such that
\[
\|\nabla \Phi  \|_{L^2(B(0,r))} \aleq \exp\brac{-C_1(r,s)\, \int_{B}(\lambda^h)_- + C_2(r,s)\, \brac{\|\nabla \Phi \|_{2,B} + \mathcal{W}_{s,\frac{2}{s}}(\Phi,B(0,1))}}
\]
where we use the notation $f_- = |\min \{f,0\}|$.

Here $\lambda^h$ is the solution of
\[
\begin{cases}
\lap \lambda^h = 0 \quad& \mbox{in $B(0,1)$},\\
\lambda^h = \lambda \quad &\mbox{on $\partial B(0,1)$}.
\end{cases}
\]
\end{lemma}
\begin{proof}
Fix $s \in (\frac{1}{2},1]$. We denote $\lambda_0 := \lambda-\lambda^h$. 

By the conformality of $\Phi$, namely by \eqref{eq:phiconformal2}, we have for any $z \in B(0,r)$
\[
|\nabla \Phi(z) |^2 = e^{2 \lambda(z)} \leq e^{2 \sup_{B(0,r)} \lambda^h}\, e^{2|\lambda_0(z)|}.
\]
Thus
\begin{equation}\label{eq:nphiz}
\|\nabla \Phi\|_{L^2(B(0,r)} \leq e^{2 \sup_{B(0,r)} \lambda^h} \int_{B(0,1)} e^{2|\lambda_0(z)|}
\end{equation}
We analyze the two factors on the right-hand side of \eqref{eq:nphiz}. For the $\lambda^h$-term we argue as follows: from Lemma~\ref{la:lambdahestK} we get 
\begin{equation}\label{eq:supzlambdah}
\sup_{z \in B(0,r)} \lambda^h(z) \leq C(r,s) \brac{-\int_{B(0,1)} (\lambda^h)_- + [\tilde{e}_1]^2_{W^{s,\frac{2}{s}}(B(0,1))}+ [\tilde{e}_2]^2_{W^{s,\frac{2}{s}}(B(0,1))} + \|\nabla \Phi \|_{L^2} }
\end{equation}
where $\tilde{e}_1$, $\tilde{e}_2$ are chosen from Theorem~\ref{th:frlifting}. In view of \eqref{eq:frliftest} estimate \eqref{eq:supzlambdah} becomes
\begin{equation}\label{eq:supzlambdah2}
\sup_{z \in B(0,r)} \lambda^h(z) \leq -C_1(r,s) \int_{B(0,1)} (\lambda^h)_- +  C_2(r,s)\brac{\mathcal{W}_{s,\frac{2}{s}}(\Phi,B(0,1)) + \|\nabla \Phi \|_{L^2} }.
\end{equation}
This gives an estimate for the $\lambda^h$-quantity in \eqref{eq:nphiz}.

Regarding the $\lambda_0$-quantity in \eqref{eq:nphiz} we use that as in Lemma~\ref{la:lambdahestK}, more precisely by \eqref{eq:lambda0eq}, we can apply Wente's inequality, Proposition~\ref{pr:wente1}, to $\lambda_0$. Namely,
\[
\|\nabla \lambda_0 \|_{L^2(B(0,1))} \aleq [\tilde{e}_1]_{W^{s,\frac{2}{s}}(B(0,1)}\, [\tilde{e}_2]_{W^{s,\frac{2}{s}}(B(0,1)},
\]
From \eqref{eq:frliftest} we have in particular
\[
\|\nabla \lambda_0 \|_{L^2(B(0,1))} \leq 1.
\]
Since
\[
e^{2|\lambda_0|} \leq C(\beta)\, e^{\beta |\lambda_0|^2}
\]
we conclude that we can employ Moser-Trudinger inequality \cite{T67,M70,A88}, and have
\begin{equation}\label{eq:lambda0est}
\int_{B(0,1)} e^{2|\lambda_0|} \aleq 1.
\end{equation}
Plugging \eqref{eq:lambda0est} and \eqref{eq:supzlambdah2} into \eqref{eq:nphiz} we conclude.
\end{proof}

Now we are ready to prove our main result.
\begin{proof}[Proof of Theorem~\ref{th:fracms}]
Let $\Phi_k \in C^\infty(\overline{B(0,1)},\R^3)$ be a $W^{1,2}$-weakly converging sequence of conformal immersions with
\[
\sup_{k \in \N} \mathcal{W}_{s,\frac{2}{s}}(\Phi_k;B(0,1)) < \eps_s.
\]
Then we find $\lambda_k \in C^\infty(\overline{B(0,1)})$ and orthonormal basis $(e_{1;k},e_{2;k},u_k)$ such that
\[
\partial_\alpha \Phi_k = e^{\lambda_k} e_{\alpha;k} \quad \alpha = 1,2
\]
We split $\lambda_k = \lambda_k^h + \lambda_{k,0}$ as in Lemma~\ref{la:nablaphiest} and consider two cases:

Firstly, we consider the collapsing case, that is we assume
\[
\sup_{k} \int_{B(0,1)} (\lambda_k^h)_- = \infty.
\]
In this case we obtain from Lemma~\ref{la:nablaphiest},
\[
\liminf_{k \to \infty} \|\nabla \Phi_k\|_{L^2(B(0,r))} = 0.
\]
Weak convergence of $\Phi_k \to \Phi$ in $W^{1,2}$ then implies that $\Phi$ is a constant map.

Assume now this is not the case, that is assume
\[
\sup_{k} \int_{B(0,1)} (\lambda_k^h)_- < \infty.
\]
In that case we get from Lemma~\ref{la:nablaphiest} and Lemma~\ref{la:harmonic1} applied to $-\lambda^h$ 
\[
\sup_{k} \|\lambda_k^h\|_{L^\infty(B(0,r))} < \infty \quad \mbox{for any $r \in (0,1)$}.
\]
In particular. since $\lambda_k$ is harmonic, for any $r \in (0,1)$ we have 
\begin{equation}\label{eq:lambdahw12}
\sup_{k} \|\lambda_k^h\|_{W^{1,2}(B(0,r))} < \infty \quad \mbox{for any $r \in (0,1)$}.
\end{equation}
Moroever, from Theorem~\ref{th:lifting} we have that we have frames $\tilde{e}_{\alpha;k} = P_{\alpha \beta;k}(\theta_k) e_{\beta;k}$
\begin{equation}\label{eq:ewsp}
[\tilde{e}_{1;k}]_{W^{s,\frac{2}{s}}(B(0,1)}^{\frac{2}{s}} + [\tilde{e}_{2;k}]_{W^{s,\frac{2}{s}}(B(0,1)}^{\frac{2}{s}} + \|\langle \tilde{e}_1,\nabla \tilde{e}_2 \rangle\|_{L^2(B(0,1))} \aleq \mathcal{W}_{s,\frac{2}{s}}(\Phi,B(0,1)).
\end{equation}
Also in view of Lemma~\ref{la:thetaconfparam} 
\begin{equation}\label{eq:lambda0kdecomp}
- \nabla^\perp \lambda_{0;k} -   \nabla^\perp\lambda_k^h= \langle \tilde{e}_{1;k},\nabla \tilde{e}_{2;k} \rangle_{\R^3} + \nabla \theta_k .
\end{equation}
which taking the curl implies that $\lambda_{0;k}$ satisfies 
\[
\begin{cases}
\lap \lambda_{0;k} = \langle \nabla^\perp \tilde{e}_{1;k},\nabla \tilde{e}_{2;k} \rangle_{\R^3} \quad &\mbox{in $B(0,1)$}\\
\lambda_{0;k} = 0 \quad &\mbox{on $\partial B(0,1)$}.
\end{cases}
\]
We can apply Wente's theorem, Proposition~\ref{pr:wente1}, to obtain
\begin{equation}\label{eq:lambda0kw12}
\|\nabla \lambda_{0;k} \|_{L^2(B(0,1))} \aleq \brac{\mathcal{W}_{s,\frac{2}{s}}(\Phi,B(0,1))}^{\frac{s}{2}}.
\end{equation}
In view of the decomposition \eqref{eq:lambda0kdecomp}, the estimates \eqref{eq:lambdahw12}, \eqref{eq:ewsp}, and \eqref{eq:lambda0kw12} imply a locally uniform $\dot{W}^{1,2}$-bound on $\theta$, and in particular we get a locally uniform  $W^{1,2}$-bound on $P_k = P(\theta_k)$,
\[
\sup_{k \in \N} \|\nabla P_k\|_{L^2(B(0,r))} < \infty \quad \mbox{for any $r \in (0,1)$}.
\]
This finally leads to the fact that $e_{\alpha;k}  = P^t_{\alpha \beta} \tilde{e}_{\beta}$ is locally bounded
\[
\sup_{k \in \N} \brac{ [e_{1;k}]_{W^{s,\frac{2}{s}}(B(0,r)} + [e_{2;k}]_{W^{s,\frac{2}{s}}(B(0,r))}} < \infty \quad \mbox{for any $r \in (0,1)$}.
\]
In particular up to taking subsequences we find on any $B(0,r)$ an $L^2(B(0,r))$-strongly converging subsequence of $e_{\alpha,k}$ to an orthnormal system $e_{\alpha} \in  W^{s,\frac{2}{s}}(B(0,r),\S^2)$.

By \eqref{eq:lambdahw12} and \eqref{eq:lambda0kw12} we also find that $e^{\lambda_k}$ converges almost everywhere to some $e^{\lambda}$ with 
\[
\|\lambda\|_{W^{1,2}(B(0,r))} < \infty \quad \mbox{for any $r \in (0,1)$}.
\]
Since $|\nabla \Phi_k|^2 = e^{2\lambda_k}$ and $\Phi_k$ is bounded in $W^{1,2}$ we get also weak convergence of $e^{\lambda_k}$ in $L^2$ to $e^{\lambda}$.

In particular, we can pass to the limit in the equation
\[
\partial_\alpha \Phi_k = e^{\lambda_k} \tilde{e}_{\alpha;k} \quad \mbox{in $B(0,r)$}.
\]
This way we find
\[
\partial_\alpha \Phi = e^{\lambda}\, \tilde{e}_{\alpha} \quad \mbox{almost everywhere in B(0,1)}.
\]
Since $\lambda \in L^1_{loc}$ we have in particular that $\Phi$ is almost everywhere an immersion.
\end{proof}

\section{The fractional normal curvature quantity controls the Sobolev norm: Proof of Proposition~\ref{pr:uwusmalluwsmall}}\label{s:proofuwusmall}
We begin with an easy lemma
\begin{lemma}\label{la:uw1}
Let $B \subset \R^2$ be a ball and $u \in C^\infty(\overline{B},\S^2)$. Then we have
\[
[u]_{W^{s,\frac{2}{s}}(B)} \aleq [u]_{W^{s,\frac{2}{s}}(B)}^2 + \brac{\int_{B}\int_{B}\frac{|u(x)\wedge u(y)|^{\frac{2}{s}}}{|x-y|^{4}}\, dx\, dy}^{\frac{s}{2}}
\]
The constants are independent of the specific ball.
\end{lemma}
\begin{proof}
Since $|u| = 1$, by Lagrange identity,
\[
\begin{split}
|u(x)-u(y)| \aleq& |u(x) \wedge (u(x)-u(y)| + |u(x) \cdot (u(x)-u(y))|\\
=& |u(x) \wedge u(y)| + \frac{1}{2}|u(x) - u(y)|^2
\end{split}
\]
In the last step we used that $\langle u(x)+u(y), u(x)-u(y) \rangle = 1-1 = 0$. Thus,
\[
[u]_{W^{s,\frac{2}{s}}(B)} \aleq [u]_{W^{\frac{s}{2},\frac{4}{s}}(B)}^2  + \brac{\int_{B}\int_{B}\frac{|u(x)\wedge u(y)|^{\frac{2}{s}}}{|x-y|^{4}}\, dx\, dy}^{\frac{s}{2}}
\]
From Sobolev embedding, Proposition~\ref{pr:sob}, we obtain
\[
[u]_{W^{s,\frac{2}{s}}(B)} \aleq [u]_{W^{s,\frac{2}{s}}(B)}^2 + \brac{\int_{B}\int_{B}\frac{|u(x)\wedge u(y)|^{\frac{2}{s}}}{|x-y|^{4}}\, dx\, dy}^{\frac{s}{2}}.
\]
\end{proof}

\begin{proof}[Proof of Proposition~\ref{pr:uwusmalluwsmall}]
From Lemma~\ref{la:uw1} we get in particular that for $f,g: [0,\infty) \to [0,\infty)$ defined as
\[
f(r) := [u]_{W^{s,\frac{2}{s}}(B(0,r))}
\]
\[
g(r) := \brac{\int_{B(0,r)}\int_{B(0,r)}\frac{|u(x)\wedge u(y)|^{\frac{2}{s}}}{|x-y|^{4}}\, dx\, dy}^{\frac{s}{2}}
\]
we have the estimate
\[
f(r) \leq C\, (f(r))^2 + g(r).
\]
Considering the roots and asymptotics of the $C x^2 + g(r) - x$ we conclude that if
\[
F_1(r) := \frac{1}{2C} - \sqrt{\frac{1}{4C^2} - \frac{g(r)}{C}},\quad F_2(r) := \frac{1}{2C} + \sqrt{\frac{1}{4C^2} - \frac{g(r)}{C}}
\]
are real numbers then 
\begin{equation}\label{eq:fr}
\mbox{either } f(r) \in [0,F_1(r)] \quad \mbox{or} \quad f(r) \in [F_2(r),\infty).
\end{equation}
Now we show the claim for $B(0,1)$, by a scaling and translation argument it then holds for any ball $B$. Assume that $g(1) \leq \eps_0$, then $g(r) \leq \eps_0$ for all $r \in (0,1]$. In particular, if we choose $\eps_0$ small enough we have that $F_1$ and $F_2$ are real, and moreover,
\begin{equation}\label{eq:Frgr}
F_1(r) \aleq g(r) \quad \mbox{for any $r \in (0,1]$}.
\end{equation}
Since on the other hand $u \in W^{s,\frac{2}{s}}(B(0,1))$, we get from the absolute continuity of the integral
\[
\lim_{r \to 0} f(r) = 0.
\]
Since $F_2(r) - F_1(r) > c$ uniformly for any $r \in [0,1]$ we conclude that
\[
f(r) \leq F_1(r) \quad \mbox{for all $r \in [0,1]$}.
\]
In view of \eqref{eq:Frgr} we obtain the claim for the ball $B(0,1)$. 
\end{proof}

\section{Estimates on orthormal systems: Proof of Proposition~\ref{pr:eest}}\label{s:eest}
Proposition~\ref{pr:eest} is essentially a consequence of the following global estimate.
\begin{theorem}\label{th:eestgen}
For $n \geq 2$. For any $s \in (\frac{1}{2},1)$, if $e \perp u$, $|e| = |u| = 1$ almost everywhere, then
\[
[e]_{W^{s,\frac{n}{s}}(\R^n)} \aleq \|\Pi^\perp(u) \nabla e\|_{L^n(\R^n)} + [u]_{W^{s,\frac{n}{s}}(\R^n)} + [e]_{W^{s,\frac{n}{s}}(\R^n)}\, [u]_{W^{s,\frac{n}{s}}(\R^n)}
\]
Here,
\[
\Pi^\perp(u) = I - u \otimes u
\]
is the projection onto $u^\perp$.
\end{theorem}
\begin{proof}[Proof of Theorem~\ref{th:eestgen}]
We will write $u^\perp(x) := \Pi^\perp u(x)$.
Since $|u| = 1$,
\[
\begin{split}
|e(x)-e(y)| \aleq& |u(x) \cdot (e(x)-e(y))| + |\Pi^\perp(u(x)) (e(x)-e(y))|\\
\leq & |u(x) \cdot e(y)| + |u(x) \wedge (e(x)-e(y))| \\
\leq & 2\|e_1\|_{\infty}\, |u(x) -u(y)| + \frac{1}{2} |\brac{u^\perp(x)+u^\perp(x)} (e(x)-e(y))|\\ 
\end{split}
\]

\[
[e]_{W^{s,\frac{n}{s}}(\R^n)} \aleq [u]_{W^{s,\frac{n}{s}}(\R^n)} + \brac{\int_{\R^n} \int_{\R^n} \frac{|\brac{u^\perp(x)+u^\perp(y)}(e(x)-e(y))|^\frac{n}{s}}{|x-y|^{2n}}\, dx\, dy}^{\frac{s}{n}}.
\]
Since $e = \lap^{-1} \partial_\alpha \partial_\alpha e$ we can write
\[
e(x) - e(y) = \int_{\R^n} \brac{\frac{(x-z)^\alpha}{|x-z|^n} - \frac{(y-z)^\alpha}{|y-z|^n}}\, \partial_\alpha e(z)\, dz. 
\]
Here we use (mainly for technical reasons) $n \geq 2$.

Consequently
\[
\begin{split}
&|(u^\perp(x) + u^\perp(y)) (e_1(x)-e_1(y))|\\
 \leq & 2\left |\int_{\R^n} \brac{\frac{(x-z)^\alpha}{|x-z|^n} - \frac{(y-z)^\alpha}{|y-z|^n}}\, u^\perp(z) \partial_\alpha e(z)\, dz \right|\\
&+ \left |\int_{\R^n} \brac{\frac{(x-z)^\alpha}{|x-z|^n} - \frac{(y-z)^\alpha}{|y-z|^n}}\, (u^\perp(x) + u^\perp(y)-2u^\perp(z)) \partial_\alpha e_1(z)\, dz \right|
\end{split}
\]
Now we get from a Sobolev-type embedding, namely Proposition~\ref{pr:sob1},
\[
\begin{split}
&\brac{\int_{\R^n} \int_{\R^n} \left |\int_{\R^n} \brac{\frac{(x-z)^\alpha}{|x-z|^n} - \frac{(y-z)^\alpha}{|y-z|^n}}\, u^\perp(z) \partial_\alpha e(z)\, dz \right|^\frac{n}{s} \frac{ dx\, dy}{|x-y|^{2n}}}^{\frac{s}{n}}\\
 \aleq &\|u^\perp \nabla e\|_{L^n(\R^n)}.
\end{split}
\]
Moreover, from commutator-type estimates,namely by Proposition~\ref{pr:bigcommie}, we have
\[
\begin{split}
&\brac{\int_{\R^n} \int_{\R^n} \left |\int_{\R^n} \brac{\frac{(x-z)^\alpha}{|x-z|^n} - \frac{(y-z)^\alpha}{|y-z|^n}}\, (u^\perp(x) + u^\perp(y)-2u^\perp(z)) \partial_\alpha e_1(z)\, dz \right|}^{\frac{s}{n}}\\
 \aleq& [u]_{W^{s,\frac{n}{s}(\R^n)}}\, [e]_{W^{s,\frac{n}{s}}(\R^n)}.
\end{split}
\]
We thus conclude.
\end{proof}

We also have the analogue of this statement for the fractional Laplacian which we record here. This is much simpler to prove than Theorem~\ref{th:eestgen}.
\begin{proposition}\label{pr:evsuest:old}
For $n \geq 2$.
For any $s \in [\frac{1}{2},1)$, if $e \perp u$, $|e| = |u| = 1$ almost everywhere, then
\[
\|\laps{s} e\|_{L^{\frac{n}{s}}(\R^n)} \aleq \|\Pi^\perp(u) \nabla e\|_{L^n(\R^n)} + \|\laps{s} u\|_{L^{\frac{n}{s}}(\R^n)} \|\laps{s} e\|_{L^{\frac{n}{s}}(\R^n)} + \|\laps{s} u\|_{L^{\frac{n}{s}}(\R^n)}.
\]
Here,
\[
\Pi^\perp(u) = I - u \otimes u
\]
is the projection onto $u^\perp$.
\end{proposition}
\begin{proof}[Proof of Proposition~\ref{pr:evsuest:old}]
Denote by $\nabla^s := \Rz\laps{s}$, where $\Rz$ is the vectorial Riesz transform. We then have by boundedness of the Riesz transform on $L^p$-spaces,
\[
\|\laps{s} e \|_{L^{\frac{n}{s}}(\R^n)} \aeq \|\nabla^s e \|_{L^{\frac{n}{s}}(\R^n)}.
\]
In particular the claim is obvious if $s =1$.

From now on assume let $s \in [\frac{1}{2}, 1)$. We denote $u^\perp := \Pi^\perp(u)$, and have
\[
|\nabla^s e|\aleq |u^\perp \nabla^s e| + |u \cdot \nabla^s e|.
\]
Using the fractional Leibniz rule, see e.g. \cite{LS18}, we have
\[
\|u \cdot \nabla^s e\|_{L^{\frac{n}{s}}} \aleq 
\|e\|_{\infty}\, \|\laps{s} u\|_{L^{\frac{n}{s}}} + \|\laps{s} e\|_{L^{\frac{n}{s}}}\, \|\laps{s} u\|_{L^{\frac{n}{s}}}.
\]
For the remaining term we use a commutator,
\[
u^\perp \nabla^s e = u^\perp \lapms{1-s} \nabla e = [u^\perp, \lapms{1-s}](\nabla e) + \lapms{1-s} (u^\perp \nabla e).
\]
Since $s < 1$ we get from Sobolev embedding, Proposition~\ref{pr:sob},
\[
\|\lapms{1-s} (u^\perp \nabla e)\|_{L^{\frac{n}{s}}} \aleq \|\langle e,\nabla e_2 \rangle\|_{2}
\]
It remains to estimate the commutator, and for some $\varphi \in C_c^\infty(\R^n)$ with $\|\varphi \|_{L^{\frac{n}{n-s}}(\R^n)} \leq 1$ we have
\[
\begin{split}
&\|[u^\perp, \lapms{1-s}](\nabla e)\|_{L^{\frac{n}{s}}}\\
 \aleq &\left |\int_{\R^n}  \brac{\laps{1-s} (u\, \lapms{1-s} \varphi)-\varphi u}\wedge \nabla^s e \right |\\
  \aleq &\|\brac{\laps{1-s} (u \lapms{1-s} \varphi)-\varphi u}\|_{L^{\frac{n}{n-s}}}\, \|\laps{s} e \|_{L^{\frac{n}{s}}}.
\end{split}
\]
Again by the fractional Leibniz rule, since $s \in [\frac{1}{2},1)$,
\[
\|\brac{\laps{1-s} (u \lapms{1-s} \varphi)-\varphi u}\|_{L^{\frac{n}{n-s}}}
\aleq \|\laps{1-s} u\|_{\frac{n}{1-s}}\, \|\lapms{1-s} \varphi\|_{\frac{n}{n-1}} \aleq \|\laps{s} u\|_{L^{\frac{n}{s}}}.
\]
This proves the claim. 
\end{proof}

\begin{proof}[Proof of Proposition~\ref{pr:eest}]
Clearly the statement is invariant under scaling and translation so we may assume that $B = B(0,1)$.

We also may assume that $e_1,e_2,u$ are extended to all of $\R^2$ as in Lemma~\ref{la:ext:1}. Observe that this conserves the property that $(e_1,e_2,u)$ form an orthonormal basis of $\R^3$ almost everywhere.

Observe that 
\[
|\Pi^\perp(u) \nabla e_i| = |\langle e_1, \nabla e_2\rangle |, \quad i =1,2.
\]
From Theorem~\ref{th:eestgen} and the estimates in Lemma~\ref{la:ext:1} we have
\[
[e_1]_{W^{s,\frac{2}{s}}(B)} +[e_2]_{W^{s,\frac{2}{s}}(\R^2)} \aleq \|\langle e_1, \nabla e_2 \rangle\|_{L^2(B)} + [u]_{W^{s,\frac{n}{s}}(B)} + [e_1]_{W^{s,\frac{n}{s}}(B)}\, [u]_{W^{s,\frac{n}{s}}(B)}
\]
The remaining claims follow easily.
\end{proof}

\section{A Wente-type estimate: Proof of Proposition~\ref{pr:wente1} and Proposition~\ref{pr:wenteour}}\label{s:proofwenteour}
The analysis in this work is based essentially on a sharp estimates for Wente-Lemma type equations. The usual Wente-Lemma has been used for a long time in geometric analysis \cite{Reshetnyak-1968,Wente69,BrC84,Tartar84,Mueller90,CLMS,T97}. 
It essentially is an estimate on solutions $\lambda_0$ to
\[
\lap \lambda_0 = \nabla^\perp a \cdot \nabla b \quad\mbox{in $B \subset \R^2$},
\]
where $\nabla^\perp = (-\partial_y,\partial_x)$. Up to boundary conditions one can control the $W^{1,2}$-norm of $\lambda_0$ in terms of the $L^2$-norm of $\nabla a$ and $\nabla b$. For a precise formulation see, e.g. \cite[Lemma A.1]{R07}.

This estimate is however not optimal in the sense that a control in $\lapv a$ and $\lapv b$ in $L^4$ suffices for the same $W^{1,2}$-estimate (a $L^{(4,2)}$-Lorentz space estimate is needed to get an $L^\infty$-control on $\lambda_0$).

Proposition~\ref{pr:wente} is a direct corollary of the following statement. We could not find it in the literature, so it might be interesting in its own right.
\begin{proposition}\label{pr:wente}
Let $B \subset \R^2$ be a ball. Let $a,b \in W^{1,2}(\B^2)$ and assume that $\lambda$ solves
\[
\begin{cases}
\lap \lambda = \langle \nabla^\perp a,\nabla b \rangle \quad& \mbox{in $B$}\\
\lambda = 0 \quad &\mbox{on $\partial B$}.
\end{cases}
\]
then we have for any extension $\tilde{a}, \tilde{b}$ such that $\tilde{a}-a$ and $\tilde{b}-b$ are constant in $B$
\begin{equation}\label{eq:wmain}
\|\nabla \lambda\|_{L^{p}(B)} \aleq \|\laps{s} \tilde{a}\|_{L^{p_1}(\R^2)}\, \|\laps{1-s} \tilde{b}\|_{L^{p_2}(\R^2)},
\end{equation}
whenever $s \in (0,1)$ and $1 < p,p_1,p_2 < \infty$ are such that 
\[
\frac{1}{p} = \frac{1}{p_1} + \frac{1}{p_2}.
\]
In particular, we get as a special case for any $s \in (\frac{1}{2},1)$,
\begin{equation}\label{eq:wwsp}
\|\nabla \lambda\|_{L^{2}(B)} \aleq [a]_{W^{s,\frac{2}{s}}(B)}\, [b]_{W^{s,\frac{2}{s}}(B)}.
\end{equation}
\end{proposition}
\begin{remark}
In terms of Lorentz spaces the argument below readily leads to the following estimate.
\[
\|\nabla \lambda\|_{L^{p,q}(B)} \aleq \|\laps{s} \tilde{a}\|_{L^{p_1,q_1}(\R^2)}\, \|\laps{1-s} \tilde{b}\|_{L^{p_2,q_2}(\R^2)},
\]
whenever $s \in (0,1)$, $1 \leq q, q_1,q_2 \leq \infty$ and $1 < p,p_1,p_2 < \infty$ are such that 
\[
\frac{1}{p} = \frac{1}{p_1} + \frac{1}{p_2},
\]
and
\[\frac{1}{q_1} + \frac{1}{q_2} = \frac{1}{q}.\]

In particular, for $\frac{1}{q_1} + \frac{1}{q_2} = 1$ we have by the embedding $\lapms{s} L^{(\frac{n}{s},1)}(\R^n) \subset L^\infty(\R^n)$,
\[
\|\lambda\|_{L^{\infty}(B)} \aleq \|\laps{s} \tilde{a}\|_{L^{4,q_1}(\R^2)}\, \|\laps{1-s} \tilde{b}\|_{L^{4,q_2}(\R^2)},
\]
For $s = 1$, compare this to \cite[Lemma IV.2]{BR14}.
\end{remark}

\begin{proof}[Proof of Proposition~\ref{pr:wente}]
The estimate \eqref{eq:wwsp} follows from \eqref{eq:wmain} by Sobolev embedding, Proposition~\ref{pr:sob} and using the extension from Lemma~\ref{la:ext:1} -- since for $s > \frac{1}{2}$ we can always find $\tilde{s} \in (0,1)$ such that both $\tilde{s}$ and $1-\tilde{s} < s$.

Now let us prove \eqref{eq:wmain}. By a duality argument and Hodge decomposition we have
\[
\|\nabla \lambda\|_{L^{p}(B)} \aleq \int_{B} \nabla \lambda \cdot \nabla \varphi,
\]
where $\varphi \in C_c^\infty(B)$ and for $p' = \frac{p}{p-1}$,
\[
\|\nabla \varphi\|_{L^{p'}(B)} \leq 1.
\]
Thus,
\[
\|\nabla \lambda\|_{L^{p}(B)} = 
\left |\int_{B} \langle \nabla^\perp a, \nabla b \rangle \varphi \right | 
= \left |\int_{\R^2} \langle \nabla^\perp \tilde{a}, \nabla \tilde{b} \rangle \varphi \right |.
\]
The domain of integration can be chosen $\R^2$ because $\varphi$ is zero outside of $B$. We now pretend that $\tilde{a} = a$ and $\tilde{b} = b$ for the sake of simplicity of notation.

Now we follow the ideas in \cite[Theorem~3.1.]{LS18} to get a sharp estimate. Namely let $a^h$, $b^h$, $\lambda^h$ be the harmonic extension of the respective function to $\R^3_+$. We denote the variables in $\R^3_+$ as $(x,t)$, $x \in \R^2$, $t \in (0,\infty)$. Then the div-curl structure gives via an integration by parts (cf. \cite[Theorem~3.1.]{LS18})
\[
\|\nabla \lambda\|_{L^{p}(B)}  \leq \int_{\R^3_+} |\nabla \varphi^h|\, |\nabla a^h|\, |\nabla b^h|.
\]
From \cite[Proposition 10.1]{LS18} we have
\[
|\nabla \varphi^h|(x,t) \aleq \mathcal{M}|\nabla \varphi|(x) + \mathcal{M}|\laps{1} \varphi|(x) \quad \mbox{for all $t > 0$, $x \in \R^n$}.
\]
Consequently, from H\"older's inequality and the maximal theorem,
\[
\begin{split}
\|\nabla \lambda\|_{L^{p}(B)}  \aleq &\|\nabla \varphi\|_{L^{p'}(\R^2)}\\
&\quad \cdot  
\left \| \brac{\int_0^\infty |t^{\frac{1}{2}-s}\nabla a^h(\cdot ,t)|^2 dt}^{\frac{1}{2}} \right \|_{L^{p_1}(\R^2)}\\ 
 &\quad\quad\cdot \left \| \brac{\int_0^\infty |t^{\frac{1}{2}-(1-s)} \nabla b^h(\cdot ,t)|^2 dt}^{\frac{1}{2}} \right \|_{L^{p_2}(\R^2)}.
\end{split}
\]
The latter two quantities can be identified as trace spaces, see \cite[Proposition 10.2, (10.11)]{LS18}. We conclude the proof of Proposition~\ref{pr:wente}.
\end{proof}

\begin{proof}[Proof of Proposition~\ref{pr:wenteour}]
From Proposition~\ref{pr:wente}, namely \eqref{eq:wwsp} we obtain
\[
\|\nabla \lambda\|_{L^2(B)} \aleq [e_1]_{W^{s,\frac{n}{s}}(B)}\, [e_2]_{W^{s,\frac{n}{s}}(B)}.
\]
Taking $\eps_s$ from Proposition~\ref{pr:eest} we then get
\[
\|\nabla \lambda\|_{L^2(B)} \aleq \|\langle e_1,\nabla e_2 \rangle \|_{L^2(B)}^2+ [u]_{W^{s,\frac{n}{s}}(B)}^{\frac{n}{s}}.
\]
\end{proof}

\appendix

\section{Fractional Sobolev spaces, gagliardo norms}
\label{s:fracsob}
Moreover we will -- a few times -- use the fractional Laplacian $\laps{s}$ and its inverse the Riesz potentials $\lapms{s} = (-\Delta)^{-\frac{s}{2}}$ , for $s \in (0,n)$. These operators are multipliers defined by the Fourier transform,
\[
\mathcal{F} (\laps{s} f)(\xi) = |\xi|^s \mathcal{F}f(\xi).
\]
and
\[
\mathcal{F} (\lapms{s} f)(\xi) = |\xi|^{-s} \mathcal{F}f(\xi).
\]
For the Riesz potential $\lapms{s}$ we will also use the potential representation
\[
\lapms{s} f(x) = c\, \int_{\R^n} |x-y|^{s-n}\, f(y)\, dy.
\]
For convenience we will refer to $\laps{0} = Id$.

Observe that the fractional Laplacian $\laps{s}$ and the Riesz potential $\lapms{s}$ act on functions defined on all of $\R^n$, while the Sobolev norm $W^{s,p}$ can be defined on any open set. On the $L^2$-scale one can characterize the Slobodeckij-Gagliardo space with the fractional Laplacian, namely we have
\[
[f]_{W^{s,2}(\R^n)} \aeq \|\laps{s} f\|_{L^2(\R^n)} \quad \mbox{on $s \in (0,1)$}.
\]
This also holds for $s =1$ and any $p \in (1,\infty)$,
\[
[f]_{W^{1,p}(\R^n)} = \|\nabla f\|_{L^p(\R^n)} \aeq \|\laps{1} f\|_{L^p(\R^n)}, 
\]
but this is only due to the commonly accepted abuse of notation that in the scale of Triebel-Lizorkin spaces $W^{s,p}(\R^n) = F^s_{pp}(\R^n) = B^s_{pp}(\R^n)$ for $s \in (0,1)$ but $W^{1,p}(\R^n) = F^s_{p2}(\R^n)$.

\begin{proposition}[Sobolev-embedding properties]\label{pr:sob}
\begin{enumerate}
We have the following embeddings estimates
\item \[
\|\laps{s} f\|_{L^p(\R^n)} \aleq [f]_{W^{s,p}(\R^n)} \quad \mbox{for $s \in (0,1)$, $p \in (1, 2]$}
\]
\[
[f]_{W^{s,p}(\R^n)}  \aleq \|\laps{s} f\|_{L^p(\R^n)} \quad \mbox{for $s \in (0,1)$, $p \in [2, \infty)$}
\]
\item For $0 \leq s_1 < s_2 <1$ if $p_1, p_2 \in (1,\infty)$ and $s_1-\frac{n}{p_1} = s_2-\frac{n}{p_2}$\footnote{this estimate is false for $s_1 = s_2$ and $p_1 = p_2 > 2$!}
\[
\|\laps{s_1} f\|_{L^{p_1}(\R^n)} \aleq [f]_{W^{s_2,p_2}(\R^n)} 
\]
and for $0 \leq s_1 \leq s_2 <n$ if $p_1, p_2 \in (1,\infty)$ and $s_1-\frac{n}{p_1} = s_2-\frac{n}{p_2}$,
\[
\|\laps{s_1} f\|_{L^{p_1}(\R^n)} \aleq \|\laps{s_2} f\|_{L^{p_2}(\R^n)}. 
\]
\item For $s \in (0,n)$, $p_1,p_2 \in (1,\infty)$ such that $p_1 = \frac{n p_2}{n-sp_2}$ we have 
\[
\|\lapms{s} f\|_{L^{p_1}(\R^n)} \aleq \|f\|_{L^{p_2}(\R^n)}
\]
\item For any ball $B \subset \R^n$ or $B = \R^n$ if $0 < s_1 \leq s_2 <1$ if $p_1, p_2 \in (1,\infty)$ such that $s_1-\frac{n}{p_1} = s_2-\frac{n}{p_2}$ 
\[
[f]_{W^{s_1,p_1}(B)} \aleq [f]_{W^{s_2,p_2}(B)}
\]
\end{enumerate}
\end{proposition}
\begin{proof}[Remarks on the proofs]
The last statement follows from the extension Lemma~\ref{la:ext:1}.

\end{proof}

An important tool for us is the following extension argument.
\begin{lemma}[Extension Lemma]\label{la:ext:1}
Let $u \in L^1(B(0,1))$ and set 
\[
v(x) := \begin{cases}
u(x) \quad &|x| < 1,\\
u(x/|x|^2) \quad &|x| > 1.\\
\end{cases}
\]
Then for any $s \in (0,1]$, and any $\Lambda > 1$,
\[
[v]_{W^{s,p}(B(0,\Lambda))} \leq C(s,p)\, \Lambda^{2\max \{0,n-sp\}} [u]_{W^{s,p}(B(0,1))}.
\]
In particular,
\[
[v]_{W^{s,\frac{n}{s}}(\R^n)} \leq C(s,p)\, [u]_{W^{s,\frac{n}{s}}(B(0,1))}.
\]
\end{lemma}
\begin{proof}
Clearly, $v \in L^1_{loc}(\R^n)$. Moreover we claim
\begin{equation}\label{eq:vvsu}
[v]_{W^{s,p}(B(4))} \aleq [u]_{W^{s,p}(B(1))}.
\end{equation}
Indeed, by splitting the integral we have
\[
[v]_{W^{s,p}(B(\Lambda))}^p \aleq [u]_{W^{s,p}(B(1))}^p + \int_{B(\Lambda) \backslash B(1)}\int_{B(\Lambda) \backslash B(1)} \frac{|v(x)-v(y)|^p}{|x-y|^{n+sp}}\, dx\, dy
+ \int_{B(\Lambda) \backslash B(1)}\int_{B(1)} \frac{|u(x)-v(y)|^p}{|x-y|^{n+sp}}\, dx\, dy
\]
For the second term we us the transformation $\tilde{x} := \frac{x}{|x|^2}$ and $\tilde{y} := \frac{y}{|y|^2}$ which leads to
\[
\begin{split}
&\int_{B(\Lambda) \backslash B(1)} \int_{B(\Lambda) \backslash B(1)} \frac{|v(x)-v(y)|^p}{|x-y|^{n+sp}}\, dx\, dy\\
=& \int_{B(1) \backslash B(1/\Lambda)}\int_{B(1) \backslash B(1/\Lambda)} \frac{|u(\tilde{x})-u(\tilde{x})|^p}{\left |\frac{\tilde{x}}{|\tilde{x}|^2}-\frac{\tilde{y}}{|\tilde{y}|^2}\right |^{n+sp}}\, \frac{d\tilde{x}}{|\tilde{x}|^{2n}}\, \frac{d\tilde{y}}{|\tilde{y}|^{2n}}\\
=& \int_{B(1) \backslash B(1/\Lambda)}\int_{B(1) \backslash B(1/\Lambda)} \frac{|u(\tilde{x})-u(\tilde{x})|^p}{\left |\frac{\tilde{x}}{|\tilde{x}|^2}-\frac{\tilde{y}}{|\tilde{y}|^2}\right |^{n+sp} |\tilde{x}|^{n+sp} |\tilde{y}|^{n+sp}}\, \frac{d\tilde{x}}{|\tilde{x}|^{n-sp}}\, \frac{d\tilde{y}}{|\tilde{y}|^{n-sp}}\\
\aleq & \Lambda^{2\max \{0,n-sp\}}
\int_{B(1) \backslash B(1/\Lambda)}\int_{B(1) \backslash B(1/\Lambda)} \frac{|u(\tilde{x})-u(\tilde{x})|^p}{\brac{\left |\frac{\tilde{x}}{|\tilde{x}|^2}-\frac{\tilde{y}}{|\tilde{y}|^2}\right | |\tilde{x}|\, |\tilde{y}|}^{n+sp}}\, d\tilde{x}\, d\tilde{y}\\
\end{split}
\]
Now observe
\[
\left |\frac{\tilde{x}}{|\tilde{x}|^2}-\frac{\tilde{y}}{|\tilde{y}|^2}\right | |\tilde{x}|\, |\tilde{y}| = |\tilde{x}-\tilde{y}|.
\]
Indeed,
\[
\left |\frac{\tilde{x}}{|\tilde{x}|^2}-\frac{\tilde{y}}{|\tilde{y}|^2}\right |^2 |\tilde{x}|^2\, |\tilde{y}|^2 = 
\brac{\frac{1}{|\tilde{x}|^2}+\frac{1}{|\tilde{y}|^2}-2\frac{\tilde{x}}{|\tilde{x}|^2}\cdot \frac{\tilde{y}}{|\tilde{y}|^2} }|\tilde{x}|^2\, |\tilde{y}|^2
= |\tilde{x}|^2+|\tilde{y}|^2 - 2 \tilde{x} \cdot \tilde{y} = |\tilde{x}-\tilde{y}|^2.
\]
Thus we have shown,
\[
\begin{split}
&\int_{B(\Lambda) \backslash B(1)} \int_{B(\Lambda) \backslash B(1)} \frac{|v(x)-v(y)|^p}{|x-y|^{n+sp}}\, dx\, dy\\
\aleq & \Lambda^{2\max \{0,n-sp\}}
\int_{B(1) \backslash B(1/\Lambda)}\int_{B(1) \backslash B(1/\Lambda)} \frac{|u(\tilde{x})-u(\tilde{x})|^p}{|\tilde{x}-\tilde{y}|^{n+sp}}\, d\tilde{x}\, d\tilde{y}\\
\end{split}
\]
For the second term we get by similar considerations
\[
\begin{split}
& \int_{B(\Lambda) \backslash B(1)}\int_{B(1)} \frac{|u(x)-v(y)|^p}{|x-y|^{n+sp}}\, dx\, dy\\
\aleq & \Lambda^{\max \{0,n-sp\}} \int_{B(1) \backslash B(1/\Lambda)}\int_{B(1)} \frac{|u(x)-u(\tilde{y})|^p}{\brac{\left |x-\frac{\tilde{y}}{|\tilde{y}|^2}\right |\, |\tilde{y}|}^{n+sp}}\, dx\, d\tilde{y}\\
\end{split}
\]
Now by geometric considerations for sphere-inversions, whenever $|x|, |\tilde{y}| < 1$,
\[
\left |x-\frac{\tilde{y}}{|\tilde{y}|^2}\right | \geq |x-\tilde{y}|,
\]
and consequently,
\begin{equation}\label{eq:21}
\left |x-\frac{\tilde{y}}{|\tilde{y}|^2}\right |\, |\tilde{y}| \geq \frac{1}{2}|x-\tilde{y}| \quad \mbox{for $|x| < 1$ and $|\tilde{y}| \in (\frac{1}{2},1)$},
\end{equation}
On the other hand, for $|y| \leq \frac{1}{2}$
\[
\left |x-\frac{\tilde{y}}{|\tilde{y}|^2}\right | \geq \frac{1}{2} \frac{1}{|\tilde{y}|}
\]
which implies that
and consequently,
\begin{equation}\label{eq:22}
\left |x-\frac{\tilde{y}}{|\tilde{y}|^2}\right |\, |\tilde{y}| \geq \frac{1}{2} \geq \frac{1}{4} |x-\tilde{y}| \quad \mbox{for $|x| < 1$ and $|\tilde{y}| \in (0,\frac{1}{2})$},
\end{equation}
From \eqref{eq:21} \eqref{eq:22} we obtain
\[
\begin{split}
& \int_{B(\Lambda) \backslash B(1)}\int_{B(1)} \frac{|u(x)-v(y)|^p}{|x-y|^{n+sp}}\, dx\, dy\\
\aleq & \Lambda^{\max \{0,n-sp\}} \int_{B(1) \backslash B(1/\Lambda)}\int_{B(1)} \frac{|u(x)-u(\tilde{y})|^p}{|x-\tilde{y}|^{n+sp}}\, dx\, d\tilde{y}\\
\end{split}
\]
The claim is now established.
\end{proof}

\section{Sobolev-embedding and Commutator-type estimates}

\begin{proposition}\label{pr:sob1}
Let $s \in [\frac{1}{2},1)$, $n \geq 2$, then
\[
\brac{\int_{\R^n}\int_{\R^n} \left |\int_{\R^n} \brac{\frac{x-z}{|x-z|^n} - \frac{y-z}{|y-z|^n}}\, f(z)\, dz \right|^{\frac{n}{s}} \frac{dx\, dy}{|x-y|^{2n}}}^{\frac{s}{n}}\, \aleq \|f\|_{L^n(\R^n)}
\]
\end{proposition}
\begin{proof}
Set $\beta := s + \eps$, where $\eps > 0$ is chosen small enough such that $2s+\eps - 1 \in (0,1)$.

For
\[
k_1(x,y,z) := \brac{|x-z|^{1-n}\frac{x-z}{|x-z|} - |y-z|^{1-n}\frac{y-z}{|y-z|}}.
\]
Let us set 
\[
A(g) := \brac{\int_{\R^n} \int_{\R^n} \left | \int_{\R^n} k_1(x,y,z)\, \laps{\beta} g(z) \, dz\right |^{\frac{n}{s}}  \frac{dx\, dy}{|x-y|^{n+s\frac{n}{s}}}}^{\frac{1}{p}}
\]
Then an argument almost verbatim to \cite[Lemma 1.2]{S16} implies
\[
A(g) \aleq [g]_{W^{2s+\eps-1,\frac{n}{s}}(\R^n)}.
\]
Since $s < 1$ we have that $\beta > 2s+\eps -1$. Thus, by Sobolev embedding, Proposition~\ref{pr:sob}, we have
\[
[g]_{W^{2s+\eps-1,\frac{n}{s}}(\R^n)} \aleq \|\laps{\beta} g\|_{L^n(\R^n)}.
\]
Applying this to $g := \lapms{\beta} f$ we find
\[
\brac{\int_{\R^n}\int_{\R^n} \left |\int_{\R^n} \brac{\frac{x-z}{|x-z|^n} - \frac{y-z}{|y-z|^n}}\, f(z)\, dz \right|^{\frac{n}{s}} \frac{dx\, dy}{|x-y|^{2n}}}^{\frac{s}{n}}
\aleq \|f\|_{L^n(\R^n)}.
\]
\end{proof}
We will also need the following commutator-type estimate.
\begin{proposition}\label{pr:bigcommie}
Let $s \in (\frac{1}{2},1)$ and $n \geq 2$ and $t \in (\frac{1}{2},1)$. Then,
\[
\begin{split}
&\brac{\int_{\R^n} \int_{\R^n} \left |\int_{\R^n} \brac{\frac{(x-z)}{|x-z|^n} - \frac{(y-z)}{|y-z|^n}}\, (f(x) + f(y)-2f(z))\, \nabla g(z)\, dz \right|^\frac{n}{s}\, \frac{dx\, dy}{|x-y|^{2n}}}^{\frac{s}{n}}\\
\aleq& [f]_{W^{t,\frac{n}{t}}(\R^n)}\, [g]_{W^{t,\frac{n}{t}}(\R^n)}
\end{split}
\]
\end{proposition}
\begin{proof}
We denote
\[
\begin{split}
T(f,g)(x,y) :=&\int_{\R^n} \brac{\frac{(x-z)}{|x-z|^n} - \frac{(y-z)}{|y-z|^n}}\, \cdot \nabla g(z)\, (f(x) + f(y)-2f(z))\, dz\\
\end{split}
\]
Writing $\nabla g = D^{\frac{1}{2}} \lapv g$, we have that
\[
T(f,g)(x,y) = \int_{\R^n} D_z^{\frac{1}{2}} \cdot \brac{\brac{\frac{(x-z)}{|x-z|^n} - \frac{(y-z)}{|y-z|^n}}\,  (f(x) + f(y)-2f(z))}\, \lapv g(z) dz\\
\]
With the integral representation
\[
D^{\frac{1}{2}} a(z) = \int_{\R^n} \brac{a(z)-a(w)}\frac{z-w}{|z-w|^{n+\frac{3}{2}}}\, dz
\]
we find
\[
D^{\frac{1}{2}} (ab)(z) = D^{\frac{1}{4}} a(z)\, b(z) + 
\int_{\R^n}(b(z) -b(w))\, a(w) \, \frac{z-w}{|z-w|^{n+\frac{3}{2}}}\, dz
\]
This implies for any $t_1 \in [0,1]$, using also the estimate from Lemma~\ref{la:kxyz3},
\[
\begin{split}
&\left |D_z^{\frac{1}{2}} \cdot \brac{\brac{\frac{(x-z)}{|x-z|^n} - \frac{(y-z)}{|y-z|^n}}\,  (f(x) + f(y)-2f(z))}\right |\\
\aleq & \left ||x-z|^{\frac{1}{2}-n}-|y-z|^{\frac{1}{2}-n} \right |\, |f(x)+f(y)-2f(z)|\\
& + \int_{\R^n}|f(z) -f(w)|\, \left ||x-w|^{1-n} - |y-w|^{1-n}  \right | \, \frac{1}{|z-w|^{n+\frac{1}{2}}}\, dz\\
& + |x-y|^{t_1} \int_{\R^n}|f(z) -f(w)|\,  \min\{|x-w|^{1-t-n}, |y-w|^{1-t-n}\}\frac{1}{|z-w|^{n+\frac{1}{2}}}\, dz
\end{split}
\]

Using the estimate (see \cite[Proposition 6.6.]{S19})
\[
|f(z)-f(w)| \aleq |z-w|^{t_2} \brac{|\mathcal{M} \laps{t_1} f(z)|+|\mathcal{M} \laps{t_2} f(w)|} \quad \mbox{for $t_2 \in [0,1)$}
\]
we then arrive at
\[
\begin{split}
&\left |D_z^{\frac{1}{2}} \cdot \brac{\brac{\frac{(x-z)}{|x-z|^n} - \frac{(y-z)}{|y-z|^n}}\,  (f(x) + f(y)-2f(z))}\right |\\
\aleq & \left ||x-z|^{\frac{1}{2}-n}-|y-z|^{\frac{1}{2}-n} \right |\, |f(x)+f(y)-2f(z)|\\
& + \int_{\R^n}\mathcal{M} \laps{t_2} f(z)\, \left ||x-w|^{1-n} - |y-w|^{1-n}  \right | \, \frac{1}{|z-w|^{n+\frac{1}{2}-t_2}}\, dw\\
& + \int_{\R^n} \mathcal{M} \laps{t_2} f(w)\, \left ||x-w|^{1-n} - |y-w|^{1-n}  \right | \, \frac{1}{|z-w|^{n+\frac{1}{2}-t_2}}\, dw\\
& + |x-y|^{t_1} \int_{\R^n}\mathcal{M} \laps{t_2} f(z)\,  \min\{|x-w|^{1-t_1-n}, |y-w|^{1-t_1-n}\}\frac{1}{|z-w|^{n+\frac{1}{2}-t_2}}\, dw\\
& + |x-t|^{t_1} \int_{\R^n}\mathcal{M} \laps{t_2} f(w)\,  \min\{|x-w|^{1-t_1-n}, |y-w|^{1-t_1-n}\}\frac{1}{|z-w|^{n+\frac{1}{2}-t_2}}\, dw\\
\end{split}
\]
By Lemma~\ref{la:xyz1}, for any $t_3 \in (0,1)$, we can estimate this further by
\[
\begin{split}
\aleq & \left ||x-z|^{\frac{1}{2}-n}-|y-z|^{\frac{1}{2}-n} \right |\, |f(x)+f(y)-2f(z)|\\
& + |x-y|^{1-t_3}\int_{\R^n}\mathcal{M} \laps{t_2} f(z)\, \left ||x-w|^{t_3-n} - |y-w|^{t_3-n}  \right | \, \frac{1}{|z-w|^{n+\frac{1}{2}-t_2}}\, dw\\
& + |x-y|^{1-t_3} \int_{\R^n} \mathcal{M} \laps{t_2} f(w)\, \left ||x-w|^{t_3-n} - |y-w|^{t_3-n}  \right | \, \frac{1}{|z-w|^{n+\frac{1}{2}-t_2}}\, dw\\
& + |x-y|^{t_1} \int_{\R^n}\mathcal{M} \laps{t_2} f(z)\,  \min\{|x-w|^{1-t_1-n}, |y-w|^{1-t_1-n}\}\frac{1}{|z-w|^{n+\frac{1}{2}-t_2}}\, dw\\
& + |x-t|^{t_1} \int_{\R^n}\mathcal{M} \laps{t_2} f(w)\,  \min\{|x-w|^{1-t_1-n}, |y-w|^{1-t_1-n}\}\frac{1}{|z-w|^{n+\frac{1}{2}-t_2}}\, dw\\
\end{split}
\]
Consequently, we arrive at
\[
\begin{split}
&T(x,y)\\
 \aleq &\int_{\R^n}   \left ||x-z|^{\frac{1}{2}-n}-|y-z|^{\frac{1}{2}-n} \right |\, |f(x)+f(y)-2f(z)|\, |\lapv g(z)| dz\\
&+|x-y|^{1-t_3}\int_{\R^n}   \int_{\R^n}\mathcal{M} \laps{t_2} f(z)\, \left ||x-w|^{t_3-n} - |y-w|^{t_3-n}  \right | \, \frac{1}{|z-w|^{n+\frac{1}{2}-t_2}}\, dw\, |\lapv g(z)| dz\\
&+|x-y|^{1-t_3}\int_{\R^n} \int_{\R^n} \mathcal{M} \laps{t_2} f(w)\, \left ||x-w|^{t_3-n} - |y-w|^{t_3-n}  \right | \, \frac{1}{|z-w|^{n+\frac{1}{2}-t_2}}\, dw \, |\lapv g(z)| dz\\
&+|x-y|^{t_1}\int_{\R^n}  \int_{\R^n}\mathcal{M} \laps{t_2} f(z)\,  \min\{|x-w|^{1-t_1-n}, |y-w|^{1-t_1-n}\}\frac{1}{|z-w|^{n+\frac{1}{2}-t_2}}\, dw\, |\lapv g(z)| dz\\
&+|x-y|^{t_1}\int_{\R^n}  \int_{\R^n}\mathcal{M} \laps{t_2} f(w)\,  \min\{|x-w|^{1-t_1-n}, |y-w|^{1-t_1-n}\}\frac{1}{|z-w|^{n+\frac{1}{2}-t_2}}\, dw\, |\lapv g(z)| dz.
\end{split}
\]
Using the integral representation of the Riesz potential, if we choose $\frac{1}{2} < t_2 < \min\{s,t\}$ (where $t$ is from the statement of the claim), this becomes
\[
T(x,y) \aleq \sum_{i=1}^5 T_i(x,y),
\]
where
\[
\begin{split}
T_1(x,y) :=& \int_{\R^n}   \left ||x-z|^{\frac{1}{2}-n}-|y-z|^{\frac{1}{2}-n} \right |\, |f(x)+f(y)-2f(z)|\, |\lapv g(z)| dz\\
T_2(x,y) :=& |x-y|^{1-t_3} \int_{\R^n}   \lapms{t_2-\frac{1}{2}} \brac{\mathcal{M} \laps{t_2} f |\lapv g|}(w)\, \left ||x-w|^{t_3-n} - |y-w|^{t_3-n}  \right | dw\\
T_3(x,y) :=& |x-y|^{1-t_3} \int_{\R^n} \mathcal{M} \laps{t_2} f(w)\, \left ||x-w|^{t_3-n} - |y-w|^{t_3-n}  \right | \, \lapms{t_2-\frac{1}{2}}|\lapv g(w)|\, dw\\
T_4(x,y) :=& |x-y|^{t_1}\int_{\R^n} \lapms{t_2-\frac{1}{2}}\brac{\mathcal{M} \laps{t_2} f |\lapv g|}(w)\,  \min\{|x-w|^{1-t_1-n}, |y-w|^{1-t_1-n}\}\, dw\\
T_5(x,y) :=& |x-y|^{t_1}\int_{\R^n}\mathcal{M} \laps{t_2} f(w)\, \lapms{t_2-\frac{1}{2}}|\lapv g|(w) \min\{|x-w|^{1-t_1-n}, |y-w|^{1-t_1-n}\}\, dw
\end{split}
\]
For $T_1$ we obtain from Proposition~\ref{pr:last},
\[
\brac{\int_{\R^n} \int_{\R^n} \frac{|T_1(x,y)|^{\frac{n}{s}}|}{|x-y|^{2n}}\, dx\, dy}^{\frac{s}{n}}  \aleq\|\lapv f\|_{L^{2n}(\R^n)}\, \|\lapv g\|_{L^{2n}(\R^n)}.
\]
Picking $t_3 \in (1-s,1)$ (in particular we may assume $t_3 > 1-t_3$) we can estimate $T_2$ and $T_3$ by Proposition~\ref{pr:pottriebelemb}, namely we get
\[
\brac{\int_{\R^n} \int_{\R^n} \frac{|T_2(x,y)|^{\frac{n}{s}}|}{|x-y|^{2n}}\, dx\, dy}^{\frac{s}{n}}  \aleq \|\lapms{t_2-\frac{1}{2}}\brac{\mathcal{M} \laps{t_2} f |\lapv g|}\|_{L^n} \aleq \|\laps{t_2} f\|_{L^{\frac{n}{t_2}}(\R^n)}\, \|\lapv g\|_{L^{2n}(\R^n)}
\]
and
\[
\brac{\int_{\R^n} \int_{\R^n} \frac{|T_3(x,y)|^{\frac{n}{s}}|}{|x-y|^{2n}}\, dx\, dy}^{\frac{s}{n}}  \aleq 
\|\mathcal{M} \laps{t_2} f \lapms{t_2-\frac{1}{2}}|\lapv g|\|_{L^n} \aleq \|\laps{t_2} f\|_{L^{\frac{n}{t_2}}(\R^n)}\, \|\lapv g\|_{L^{2n}(\R^n)}.
\]
Picking $t_1 \in (s,1)$ we get from Lemma~\ref{la:minguyz} 
\[
\brac{\int_{\R^n} \int_{\R^n} \frac{|T_4(x,y)|^{\frac{n}{s}}|}{|x-y|^{2n}}\, dx\, dy}^{\frac{s}{n}}  \aleq 
\|\lapms{t_2-\frac{1}{2}}\brac{\mathcal{M} \laps{t_2} f |\lapv g|}\|_{L^n} \aleq \|\laps{t_2} f\|_{L^{\frac{n}{t_2}}(\R^n)}\, \|\lapv g\|_{L^{2n}(\R^n)}.
\]
and
\[
\brac{\int_{\R^n} \int_{\R^n} \frac{|T_5(x,y)|^{\frac{n}{s}}|}{|x-y|^{2n}}\, dx\, dy}^{\frac{s}{n}}  \aleq 
\|\mathcal{M} \laps{t_2} f \lapms{t_2-\frac{1}{2}}|\lapv g|\|_{L^n} \aleq \|\laps{t_2} f\|_{L^{\frac{n}{t_2}}(\R^n)}\, \|\lapv g\|_{L^{2n}(\R^n)}.
\]
We conclude by noting that by Sobolev embedding, \ref{pr:sob}, since $t_2 < t$ and $\frac{1}{2} < t$,
\[
\|\lapv f\|_{L^{2n}(\R^n)}\, \|\lapv g\|_{L^{2n}(\R^n)} + \|\laps{t_2} f\|_{L^{\frac{n}{t_2}}(\R^n)}\, \|\lapv g\|_{L^{2n}(\R^n)} \aleq [f]_{W^{t,\frac{n}{t}}(\R^n)}\, [g]_{W^{t,\frac{n}{t}}(\R^n)}.
\]
\end{proof}

In the above arguments we used the following two results
\begin{proposition}\label{pr:last}
Let $n \geq 2$ and set
\[
G(x,y) := \int_{\R^n}   \left ||x-z|^{\frac{1}{2}-n}-|y-z|^{\frac{1}{2}-n} \right |\, |f(x)+f(y)-2f(z)|\, |g(z)| dz.
\]
Then for any $s \in (0,1)$ we have
\[
\brac{\int_{\R^n} \int_{\R^n} \frac{|G(x,y)|^{\frac{n}{s}}}{|x-y|^{2n}}\, dx\, dy }^{\frac{n}{s}} \aleq \|\lapv f\|_{L^{2n}(\R^n)}\, \|g\|_{L^{2n}(\R^n)}.
\]
\end{proposition}
\begin{proof}
We use a combination of the arguments in \cite[Proposition 6.6.]{S19} and \cite[Proposition 6.3.]{S15}.

Let 
\[
X(x,y,z) := |f(x)+f(y)-2f(z)|.
\]
and
\[
Y(x,y,z) := \left ||x-z|^{\frac{1}{2}-n} - |y-z|^{\frac{1}{2}-n} \right |.
\]
We will estimate the product $X(x,y,z)Y(x,y,z)$ depending on the relations between $x,y,z$.

For this we make frequent use of the following estimate, see, e.g., \cite[Proposition 6.6.]{S19}.
\[
|f(x)-f(y)| \aleq |x-y|^t \brac{|\mathcal{M} \laps{t} f(x)| + |\mathcal{M} \laps{t} f(y)|}
\]
\underline{Case 1: $|x-y|\leq \frac{1}{2}|x-z|$ or $|x-y| \leq \frac{1}{2} |y-z|$.}

In this case we have $|x-z| \aeq |y-z|$, and consequently we estimate
\[
\begin{split}
X(x,y,z) \aleq& |f(x)-f(z)| + |f(y)-f(z)|\\
 \aleq& \min\{|x-z|^{\frac{1}{2}},|y-z|^{\frac{1}{2}}\}\ \brac{\mathcal{M} \lapv f(x) + \mathcal{M} \lapv f(y) + \mathcal{M} \lapv f(z)}
\end{split}
\]
On the other hand, we get by the fundamental theorem of calculus,
\[
Y(x,y,z) \aleq |x-y|\, \min\left \{|x-z|^{-\frac{1}{2}-n}, |x-y|^{-\frac{1}{2}-n} \right \}
\]
That is, we get for any $t \in (0,1]$,
\[
X(x,y,z)\, Y(x,y,z) \aleq |x-y|^{t}\, \min\left \{|x-z|^{1-t-n}, |x-y|^{1-t-n} \right \} \, \brac{\mathcal{M} \laps{t_2} f(x) + \mathcal{M} \laps{t_2} f(y) + \mathcal{M} \laps{t_2} f(z)}
\]
Picking any $t \in (s,1)$ (using that $s > \frac{1}{2}$) we obtain from Lemma~\ref{la:minguyz}
\[
\brac{\int_{\R^n} \int_{\R^n} \frac{\left |\int_{\R^n} X(x,y,z)\, Y(x,y,z) \chi_{|x-y| \leq \frac{1}{2}|x-z|}  dz \right |^{\frac{n}{s}}}{|x-y|^{2n}}\, dx\, dy }^{\frac{n}{s}} \aleq \|\lapv f\|_{L^{2n}(\R^n)}\, \|g\|_{L^{2n}(\R^n)}.
\]
and
\[
\brac{\int_{\R^n} \int_{\R^n} \frac{\left |\int_{\R^n} X(x,y,z)\, Y(x,y,z) \chi_{|x-y| \leq \frac{1}{2}|y-z|}  dz \right |^{\frac{n}{s}}}{|x-y|^{2n}}\, dx\, dy }^{\frac{n}{s}} \aleq \|\lapv f\|_{L^{2n}(\R^n)}\, \|g\|_{L^{2n}(\R^n)}.
\]

\underline{Case 2: $|x-y|\ageq \max\{|x-z|,|y-z|\}$ and $|x-z| \aeq |y-z|$.}
In this case we estimate 
\[
\begin{split}
X(x,y,z) \aleq& |f(x)-f(z)| + |f(y)-f(z)|\\
 \aleq& \min\{|x-z|^{\frac{1}{2}},|y-z|^{\frac{1}{2}}\}\ \brac{\mathcal{M} \lapv f(x) + \mathcal{M} \lapv f(y) + \mathcal{M} \lapv f(z)}\\
  \aleq& |x-y|^{\frac{1}{2}}\ \brac{\mathcal{M} \lapv f(x) + \mathcal{M} \lapv f(y) + \mathcal{M} \lapv f(z)}
\end{split}
\]
and use the estimate (for some $t \in (s,1)$),
\[
Y(x,y,z) \aleq |x-z|^{\frac{1}{2}-n} \aleq |x-y|^{t-\frac{1}{2}} |x-z|^{1-t-n} 
\aleq |x-y|^{t-\frac{1}{2}} \min\{|x-z|^{1-t-n}, |y-z|^{1-t-n} \}.
\]
This leads to the same estimate as in Case 1, and thus we obtain also for this case
\[
\brac{\int_{\R^n} \int_{\R^n} \frac{\left |\int_{\R^n} X(x,y,z)\, Y(x,y,z) \chi_{|x-y|\ageq \max\{|x-z|,|y-z|\}} \chi_{|x-z| \aeq |y-z|}  dz \right |^{\frac{n}{s}}}{|x-y|^{2n}}\, dx\, dy }^{\frac{n}{s}} \aleq \|\lapv f\|_{L^{2n}(\R^n)}\, \|g\|_{L^{2n}(\R^n)}.
\]

\underline{Case 3: $|x-y|\ageq \max\{|x-z|,|y-z|\}$ and $|x-z| \leq \frac{1}{2}|y-z|$.}

Then we estimate
\[
X(x,y,z) \leq |f(x)-f(y)| + |x-y|^{\frac{1}{2}} \brac{|\mathcal{M}\lapv f(y)| + |\mathcal{M}\lapv f(z)|}
\]
and thus get the estimate (for some $t \in (s,1)$),
\[
\begin{split}
X(x,y,z) Y(x,y,z) \aleq &
\left |y-z|^{\frac{1}{2}-n} - |x-z|^{\frac{1}{2}-n} \right | |f(x)-f(y)|\\
& + |x-y|^{\frac{1}{2}} \left ||y-z|^{\frac{1}{2}-n} - |x-z|^{\frac{1}{2}-n} \right |\, \mathcal{M} \lapv f(z)\\
& + |x-z|^{1-t-n} |x-y|^{t} \mathcal{M} \lapv f(y)\\
\end{split}
\]
Observe that in view of Proposition~\ref{pr:pottriebelemb} (since $\frac{s}{2} < \frac{1}{2}$),
\[
\begin{split}
&\int_{\R^n} \int_{\R^n} \frac{\left |\int_{\R^n} |f(x)-f(y)|\, \left |y-z|^{\frac{1}{2}-n} - |x-z|^{\frac{1}{2}-n} \right |\, g(z)   dz \right |^{\frac{n}{s}}}{|x-y|^{2n}}\, dx\, dy  \\
&=\int_{\R^n} \int_{\R^n} |f(x)-f(y)|^{\frac{n}{s}} \frac{\left |\int_{\R^n} \, \left |y-z|^{\frac{1}{2}-n} - |x-z|^{\frac{1}{2}-n} \right |\, g(z)   dz \right |^{\frac{n}{s}}}{|x-y|^{2n}}\, dx\, dy\\
\aleq&[f]_{W^{\frac{s}{2},\frac{2n}{s}}(\R^n)}^{\frac{n}{s}}\ \brac{\int_{\R^n} \int_{\R^n} \frac{\left |\int_{\R^n} \, \left |y-z|^{\frac{1}{2}-n} - |x-z|^{\frac{1}{2}-n} \right |\, g(z)   dz \right |^{\frac{2n}{s}}}{|x-y|^{n+\frac{s}{2}\frac{2n}{s}}}\, dx\, dy }^{\frac{1}{2}} \\
\end{split}
\]
With Proposition~\ref{pr:pottriebelemb} we also have (using that $0 < s-\frac{1}{2} < \frac{1}{2}$)
\[
\begin{split}
&\int_{\R^n} \int_{\R^n} \frac{\left |
|x-y|^{\frac{1}{2}} \int_{\R^n}\left ||y-z|^{\frac{1}{2}-n} - |x-z|^{\frac{1}{2}-n} \right |\, \mathcal{M} \lapv f(z)\, g(z)
dz \right |^{\frac{n}{s}}}{|x-y|^{2n}}\, dx\, dy  \\
=&\int_{\R^n} \int_{\R^n} \frac{\int_{\R^n}\left |
\left ||y-z|^{\frac{1}{2}-n} - |x-z|^{\frac{1}{2}-n} \right |\, \mathcal{M} \lapv f(z)\, g(z)
dz \right |^{\frac{n}{s}}}{|x-y|^{n+(s-\frac{1}{2})\frac{n}{s}}}\, dx\, dy  \\
\aleq & \|\mathcal{M} \lapv f\, g\|_{L^n(\R^n)} \aleq \|\lapv f\|_{L^{2n}(\R^n)}\, \|g\|_{L^{2n}(\R^n)}.
\end{split}
\]
Regarding the third term we have
\[
\begin{split}
&\int_{\R^n} \int_{\R^n} \frac{\left |
\int_{\R^n} |x-z|^{1-t-n} |x-y|^{t} \mathcal{M} \lapv f(y)\, g(z)
dz \right |^{\frac{n}{s}}}{|x-y|^{2n}}\, dx\, dy  \\
\aeq &\int_{\R^n} \int_{\R^n} |\mathcal{M} \lapv f(y)|^{\frac{n}{s}}\, |\lapms{1-t} g(x)|^{\frac{n}{s}}
|x-y|^{(t-s)\frac{n}{s}-n}\, dx\, dy  \\
\aeq &\int_{\R^n} \int_{\R^n} \lapms{(t-s)\frac{n}{s}}\brac{|\mathcal{M} \lapv f|^{\frac{n}{s}}}(x)\, |\lapms{1-t} g(x)|^{\frac{n}{s}}\, dx.
\end{split}
\]
By H\"older inequality we then obtain the same estimate as in the arguments above since (observe that $0 < \frac{1}{2}-(t-s) < s$)
\[
\|\lapms{(t-s)\frac{n}{s}}\brac{|\mathcal{M} \lapv f|^{\frac{n}{s}}} \|_{L^{\frac{s}{\frac{1}{2}-(t-s)}}(\R^n)} \aleq 
\|\brac{\mathcal{M} \lapv f}^{\frac{n}{s}} \|_{L^{2s}(\R^n)} \aleq \|\lapv f\|_{L^{2n}(\R^n)}^{\frac{n}{s}},
\]
and
\[
\|\brac{\lapms{1-t} g(x)}^{\frac{n}{s}} \|_{L^{\frac{2s}{2t-1}}(\R^n)} \aeq 
\|\lapms{1-t} g(x) \|_{L^{\frac{2n}{2t-1}}(\R^n)}^{\frac{n}{s}} \aleq \|g\|_{L^{2n}(\R^n)}.
\]

\underline{Case 4: $|x-y|\ageq \max\{|x-z|,|y-z|\}$ and $|y-z| \leq \frac{1}{2}|x-z|$.}

This case is by symmetry the same as Case 3.

\end{proof}

\begin{lemma}\label{la:minguyz}
Let $n \geq 2$, and for $t \in (0,1)$,
\[
G(x,y) := |x-y|^t \int_{\R^n} f(z)\, \min\{|x-z|^{1-t-n},|y-z|^{1-t-n}\}\, dz.
\]
Then whenever $s \in (\frac{t}{2},t)$,
\begin{equation}\label{eq:minguyz:1}
\brac{\int_{\R^n}\int_{\R^n}\frac{|G(x,y)|^{\frac{n}{s}}}{|x-y|^{2n}}\, dx\, dy}^{\frac{s}{n}} \aleq \|f\|_{L^n(\R^n)}.
\end{equation}
Also, if moreover $s > \frac{1}{2}$,
\begin{equation}\label{eq:minguyz:2}
\brac{\int_{\R^n}\int_{\R^n}\frac{|G(x,y)|^{\frac{n}{s}}\, |g(x)|^{\frac{n}{s}}}{|x-y|^{2n}}\, dx\, dy}^{\frac{s}{n}} \aleq \|f\|_{L^{2n}(\R^n)}\, \|g\|_{L^{2n}(\R^n)}.
\end{equation}
\begin{equation}\label{eq:minguyz:3}
\brac{\int_{\R^n}\int_{\R^n}\frac{|G(x,y)|^{\frac{n}{s}}\, |g(y)|^{\frac{n}{s}}}{|x-y|^{2n}}\, dx\, dy}^{\frac{s}{n}} \aleq \|f\|_{L^{2n}(\R^n)}\, \|g\|_{L^{2n}(\R^n)}.
\end{equation}

\end{lemma}
\begin{proof}
We can estimate using the representation of the Riesz potential,
\begin{equation}\label{eq:miniguyz:gsplit}
G(x,y) \aleq |x-y|^t\, \min\{ \lapms{1-t} |f|(y),  \lapms{1-t} |f|(x)\}.
\end{equation}
Thus we obtain
\[
\begin{split}
&\int_{\R^n}\int_{\R^n}\frac{|G(x,y)\, |g(x)||^{\frac{n}{s}}}{|x-y|^{2n}}\, dx\, dy \\
\aleq& \int_{\R^n} \int_{\R^n} |x-y|^{(t-s)\frac{n}{s}-n} \brac{\lapms{1-t} |f|(y)}^{\frac{n}{s}}\, |g(x)|^{\frac{n}{s}} dx\, dy\\
\aeq&\int_{\R^n} \brac{\lapms{1-t} |f|(y)}^{\frac{n}{s}}\, \lapms{(t-s)\frac{n}{2s}}\brac{|g|^{\frac{n}{s}}}(y) dy\\
\end{split}
\]
Now observe that $t-s < \frac{1}{2}$ and thus
\[
\|\lapms{(t-s)\frac{n}{2s}} |g|^{\frac{n}{s}} \|_{L^{\frac{2s}{1-2(t-s)}}(\R^n)} \aleq \||g|^{\frac{n}{s}}\|_{L^{2s}(\R^n)} \aeq \|g\|_{L^{2n}(\R^n)}^{\frac{n}{s}},
\]
and
\[
\|\brac{\lapms{1-t} |f|}^{\frac{n}{s}}\|_{L^{\frac{2s}{2t-1}}(\R^n)} \aeq 
\|\lapms{1-t} |f|\|_{L^{\frac{2n}{2t-1}}(\R^n)}^{\frac{n}{s}} \aleq \|f\|_{L^{2n}(\R^n)},
\]
so the estimate \eqref{eq:minguyz:2} follows from H\"older's inequality. Since $G$ is symmetric in $x$ and $y$ we also obtain \eqref{eq:minguyz:3}.

As for \eqref{eq:minguyz:1}, from \eqref{eq:miniguyz:gsplit} we obtain in particular
\[
G(x,y) \aleq |x-y|^t \brac{\lapms{1-t} |f|(x)\, \lapms{1-t} |f|(y)}^{\frac{1}{2}}.
\]
Consequently,
\[
\begin{split}
&\int_{\R^n}\int_{\R^n}\frac{|G(x,y)|^{\frac{n}{s}}}{|x-y|^{2n}}\, dx\, dy \\
\aleq& \int_{\R^n} \int_{\R^n} |x-y|^{(t-s)\frac{n}{s}-n} \brac{\lapms{1-t} |f|(x)}^{\frac{n}{2s}}\, \brac{\lapms{1-t} |f|(y)}^{\frac{n}{2s}} dx\, dy\\
\aeq& \int_{\R^n} \brac{\lapms{(t-s)\frac{n}{2s}}\brac{\brac{ \lapms{1-t} |f|}^{\frac{n}{2s}}}}^2.
\end{split}
\]
Now we have (here we use that $2s > t$)
\[
\begin{split}
\|\lapms{(t-s)\frac{n}{2s}} \brac{\lapms{1-t} |f|}^{\frac{n}{2s}} \|_{L^2(\R^n)} \aleq 
\|\brac{\lapms{1-t} |f|}^{\frac{n}{2s}} \|_{L^{\frac{2s}{t}}(\R^n)} \aleq \|\lapms{1-t} |f|\|_{L^{\frac{n}{t}}(\R^n)}^{\frac{n}{2s}} \aleq \|f\|_{L^n(\R^n)}^{\frac{n}{2s}}.
\end{split}
\]
We conclude.
\end{proof}

Lastly, we used in the proof of Proposition~\ref{pr:bigcommie} and Proposition~\ref{pr:last} the following estimate 
\begin{proposition}
\label{pr:pottriebelemb}
Let
\[
T(h)(x,y) := \int_{\R^n} \left ||x-z|^{t-n} - |y-z|^{t-n}\right | h(z)\, dz.
\]
Then, for $0 < s < t < 1$ and any $q,p \in (1,\infty)$ such that 
\[
s-\frac{n}{q} = t - \frac{n}{p}
\]
we have
\[
\brac{\int_{\R^n}\int_{\R^n} \frac{\left |T(h)(x,y)\right |^q}{|x-y|^{n+sq}}\, dx\, dy}^{\frac{1}{q}} \aleq \|h\|_{L^p(\R^n)}.
\] 
\end{proposition}
\begin{proof}
This follows from Proposition~\ref{pr:rieszpottriebel} together with the Sobolev-inequality for Triebel spaces, \eqref{eq:triebelsob}.
\end{proof}

\section{On Littlewood Paley-Decomposition and Triebel Spaces}\label{s:lp}
We give only a short introduction, for more details we refer to \cite{RS96} or \cite{G14}.

For $j \in \Z$, we denote the $j$-th Littlewood Paley ``projection'' for a function $f$ as $f_j$. This operator can be represented as
\[
f_j(x) := \int_{\R^n} 2^{jn} p(2^j(z-x))\, f(z)\, dz,
\]
for some $p \in \Sw(\R^n)$ whose Fourier transform is supported on an annulus, $\supp \mathcal{F}p \subset B(0,2) \backslash B(0,\frac{1}{2})$. Moreover $p$ can be chosen such that for any $f$,
\[
f_j = \sum_{k=j-1}^{j+1} (f_{j})_k 
\]
which somewhat justifies the term ``projection'' for $f_j$.

\begin{lemma}\label{la:lapmsfj}
For any $p \in (1,\infty)$, $t \geq 0$,
\[
\|\lapms{t} f_j \|_{L^p} \aleq 2^{-jt} \|f_j \|_{L^p}
\]
\end{lemma}
\begin{proof}
Since the Fourier-transform of $p$ is supported on an annulus, we can write $p = \laps{t} q$ for some Schwartz function $q$, with Fourier transform $\mathcal{F} q = |\xi|^{-s} \mathcal{F} p$. Consequently,
\[
\lapms{t} f_j  = 2^{-jt}\, q_j \ast f,
\]
for 
\[
q_j(x) = 2^{jn} q \brac{2^jx}.
\]
We obtain the desired inequality now by Young's inequality for convolutions.
\end{proof}

For $s \in \R$ the (homogeneous) Triebel- or Besov-spaces $B^s_{p,p} (\R^n) \equiv F^{s}_{p,p}(\R^n)$ are given by
\begin{equation}\label{eq:triebel}
[f]_{\dot{F}^s_{p,p}(\R^n)} = \brac{\sum_{j \in \Z} 2^{jsp} \|f_j\|_{L^p(\R^n)}^p}^{\frac{1}{p}}
\end{equation}
For $s \in (0,1)$ we have an identification with the fractional Sobolev space, namely
\begin{equation}\label{eq:wspident}
[f]_{\dot{F}^s_{p,p}(\R^n)} \aeq [f]_{W^{s,p}(\R^n)} \quad \mbox{for $s \in (0,1)$, $p \in (1,\infty)$}.
\end{equation}
As a consequence of Sobolev embedding for Triebel spaces, \cite{J77}, we have 
\begin{equation}\label{eq:triebelsob}
[f]_{\dot{F}^s_{p,p}(\R^n)} \aleq \|f\|_{L^q(\R^n)} \quad  \mbox{for $s < 0$, $p = \frac{nq}{n-sq}$ if $p,q \in (1,\infty)$}
\end{equation}

\begin{proposition}\label{pr:rieszpottriebel}
Let
\[
T(h)(x,y) := \int_{\R^n} \left ||x-z|^{t-n} - |y-z|^{t-n}\right | h(z)\, dz.
\]
Then, for $0 < s < t < 1$ and any $p \in (1,\infty)$,
\[
\brac{\int_{\R^n}\int_{\R^n} \frac{\left |T(h)(x,y)\right |^p}{|x-y|^{n+sp}}\, dx\, dy}^{\frac{1}{p}} \aleq [h]_{\dot{F}^{s-t}_{p,p}(\R^n)}
\] 
\end{proposition}
\begin{remark}
Observe that the estimate is much easier if the absolute value
\[
\left ||x-z|^{t-n} - |y-z|^{t-n}\right | 
\]
is replaced by 
\[
\brac{|x-z|^{t-n} - |y-z|^{t-n}} 
\]
Indeed, in that case $T(h)(x,y) = c\brac{\lapms{t} h(x) -\lapms{t} h(y)}$ and thus
\[
\brac{\int_{\R^n}\int_{\R^n} \frac{\left |T(h)(x,y)\right |^p}{|x-y|^{n+sp}}\, dx\, dy}^{\frac{1}{p}} \aeq [\lapms{s} h]_{W^{t,p}(\R^n)}.
\] 

\end{remark}

\begin{proof}[Proof of Proposition~\ref{pr:rieszpottriebel}]
We follow an argument similar to \cite[Proof of Lemma~1.2]{S16}.
We will use two estimates for $T$. Firstly,
\[
T(h)(x,y) \aleq \lapms{t} h(x) + \lapms{t} h(y).
\]
Secondly, as in \cite[Lemma 6.7.]{S19}, (here $t <1$ is used)
\[
T(h)(x,y) \aleq |x-y|^t \brac{\mathcal{M}h(x) + \mathcal{M}h(y)}
\]
Let $h_j$ be the $j$th Littlewood-Paley projection of $h$. Then the above estimates lead to two estimates for
\[
\tilde{I}_{j,k} := \brac{\int_{\R^n} \int_{\R^n} \chi_{|x-y| \aeq 2^{-k}} \frac{|T(h_j)(x,y)|^p}{|x-y|^{n+sp}} \, dx\, dy}^{\frac{1}{p}}.
\]
From the first estimate we get, in view of Lemma~\ref{la:lapmsfj},
\begin{equation}\label{eq:Ijk1}
\tilde{I}_{j,k} \aleq 2^{ks} \|\lapms{t} h_j\|_{L^p(\R^n)} \aleq 2^{ksp} 2^{-jt} \|h_j\|_{L^p(\R^n)} = 2^{(k-j)s}\  2^{j(s-t)} \|h_j\|_{L^p(\R^n)}.
\end{equation}
From the second estimate we get
\begin{equation}\label{eq:Ijk2}
\tilde{I}_{j,k} \aleq 2^{-k(t-s)} \|\mathcal{M} h_j \|_{L^p(\R^n)} \aleq 
2^{(j-k)(t-s)}\  2^{j(s-t)} \|h_j \|_{L^p(\R^n)}
\end{equation}
Now we have
\[
\begin{split}
&\int_{\R^n} \int_{\R^n} \frac{|T(h)(x,y)|^p}{|x-y|^{n+sp}} \, dx\, dy\\
\aleq &
\sum_{k,j \in \Z} \int_{\R^n} \int_{\R^n} \chi_{|x-y| \aeq 2^{-k}} \frac{|T(h)(x,y)|^{p-1}\, |T(h_j)(x,y)|}{|x-y|^{n+sp}} \, dx\, dy\\
\aleq &
\sum_{k \in \Z} \brac{ 
\brac{\int_{\R^n} \int_{\R^n} \chi_{|x-y| \aeq 2^{-k}} \frac{|T(h)(x,y)|^{p}}{|x-y|^{n+sp}} \, dx\, dy}^{\frac{p-1}{p}}
\sum_j \brac{\int_{\R^n} \int_{\R^n} \chi_{|x-y| \aeq 2^{-k}} \frac{|T(h_j)(x,y)|^{p}}{|x-y|^{n+sp}} \, dx\, dy}^{\frac{1}{p}}}\\
\aleq&
\brac{\int_{\R^n} \int_{\R^n} \frac{|T(h)(x,y)|^{p}}{|x-y|^{n+sp}} \, dx\, dy}^{\frac{p-1}{p}}
\brac{\sum_k \brac{\sum_j \brac{\int_{\R^n} \int_{\R^n} \chi_{|x-y| \aeq 2^{-k}} \frac{|T(h_j)(x,y)|^{p}}{|x-y|^{n+sp}} \, dx\, dy}^{\frac{1}{p}}}^p}^{\frac{1}{p}}\\
\end{split}
\]
That is,
\[
\int_{\R^n} \int_{\R^n} \frac{|T(h)(x,y)|^p}{|x-y|^{n+sp}} \, dx\, dy
\aleq 
\sum_k \brac{\sum_j \tilde{I}_{j,k}}^p \aleq \sum_{k \in \Z} \brac{\sum_{j \geq k} \tilde{I}_{j,k}}^p + \sum_{k \in \Z} \brac{\sum_{j \leq k} \tilde{I}_{j,k}}^p 
\]
Applying \eqref{eq:Ijk1}, then Jensen's inequality, then Fubini we find
\[
\begin{split}
\sum_{k \in \Z} \brac{\sum_{j \geq k} \tilde{I}_{j,k}}^p \aleq &
\sum_{k \in \Z} \brac{\sum_{j \geq k} 2^{(k-j)s}\  2^{j(s-t)} \|h_j\|_{L^p(\R^n)} }^p\\
\aleq &\sum_{k \in \Z} \sum_{j \geq k} 2^{(k-j)s}\  \brac{2^{j(s-t)} \|h_j\|_{L^p(\R^n)} }^p\\
=&\sum_{j \in \Z} \brac{2^{j(s-t)} \|h_j\|_{L^p(\R^n)} }^p \sum_{k \leq j} 2^{(k-j)s}\  \\
\aeq &\sum_{j \in \Z} \brac{2^{j(s-t)} \|h_j\|_{L^p(\R^n)} }^p \aeq [h]_{\dot{F}^{s-t}_{p,p}}^p.
\end{split}
\]
In the last step we used the identification of Triebel spaces \eqref{eq:triebel}.

In the same spirit, from \eqref{eq:Ijk2}, (here we use that $s < t$)
\[
\begin{split}
\sum_{k \in \Z} \brac{\sum_{j \leq k} \tilde{I}_{j,k}}^p \aleq &
\sum_{k \in \Z} \brac{\sum_{j \leq k} 2^{(j-k)(t-s)}\  2^{j(s-t)} \|h_j \|_{L^p(\R^n)}}^p \\
\aleq & \sum_{k \in \Z} \sum_{j \leq k} 2^{(j-k)(t-s)}\  \brac{2^{j(s-t)} \|h_j \|_{L^p(\R^n)}}^p\\
\aeq&[h]_{\dot{F}^{s-t}_{p,p}}^p.
\end{split}
\]
The claim follows.
\end{proof}

\section{Estimates on kernels}

\begin{lemma}\label{la:xyz1}
Let $n \geq 2$. We have for any $t_1,t_2 \in (0,1)$,
\[
||x-z|^{1-n}-|y-z|^{1-n}| \aleq |x-y|^{1-t_1} \,||x-z|^{t_1-n}-|y-z|^{t_1-n}| + |x-y|^{1-t_2}\, \min\{|x-z|^{t_2-n},|y-z|^{t_2-n}\}
\]
\end{lemma}
\begin{proof}
If $|x-y| < \frac{1}{2}|x-z|$ or $|x-y| < \frac{1}{2} |y-z|$ we have
\[
||x-z|^{1-n}-|y-z|^{1-n}| \aleq |x-y|^{1-t} \, \min\{|x-z|^{t-n},|y-z|^{t-n}\}
\]
So we may assume from now on that $|x-z|, |y-z| \ageq |x-y|$.
If $|x-z| \aeq |y-z|$ then the claim follows. If $|x-z| \ll |y-z|$ we have that 
\[
||x-z|^{1-n}-|y-z|^{1-n}| \aeq |x-z|^{1-n}
\]
and
\[
||x-z|^{t-n}-|y-z|^{t-n}| \aeq |x-z|^{t-n}
\]
so also in this case the claim follows.
The same argument works also for $|y-z| \ll |x-z|$.
\end{proof}

\begin{lemma}\label{la:kxyz3}
For any $t \in [0,1]$, $m \geq 2$, Then for almost all $x,y,z \in \R^n$,
\[
\left | \frac{x-z}{|x-z|^m} - \frac{y-z}{|y-z|^m} \right | \aleq \left | |x-z|^{1-t-m} - |y-z|^{1-t-m} \right | + |x-y|^t\, \min\{|x-z|^{1-m},|y-z|^{1-m}\}.
\]
\end{lemma}
\begin{proof}
First we treat the \underline{case $|x-z| \leq \frac{1}{2} |y-z|$}, the case $|y-z| \leq \frac{1}{2} |x-z|$ is analogous.

In this case, we have
\[
\left | \frac{x-z}{|x-z|^m} - \frac{y-z}{|y-z|^m} \right | \aeq |x-z|^{1-n}
\]
and
\[
\left | |x-z|^{1-m} - |y-z|^{1-m} \right | \aeq |x-z|^{1-m}
\]
so the claim follows.

In the remaining case we have $|x-z| \aeq |y-z|$, which implies that 
\[
|x-y| \aleq |x-z|, \quad |x-y| \aleq |y-z|.
\]
In this case we get by the fundamental theorem of calculus,
\[
\left | \frac{x-z}{|x-z|^m} - \frac{y-z}{|y-z|^m} \right | \aleq |x-y| \max\{|x-z|^{-m},|y-z|^{-m}\} \aleq |x-y|^t \min\{|x-z|^{1-t-m},|y-z|^{1-t-m}\}.
\]
\end{proof}

\section{On harmonic functions}
The following is a well-known result on harmonic functions whose argument we repeat for the convenience of the reader.
\begin{lemma}\label{la:harmonic1}
Let $f$ be harmonic in $B(0,1) \subset \R^n$, then for any $K \subset \subset B$ we have a constants $C_1(K)$, $C_2(K) > 0$ so that 
\[
\sup_{K} f \leq C_1(K) \int_{B(0,1)} f_+ - C_2(K) \int_{B(0,1)} f_-
\]
where we use the notation $f_+ = \max\{f,0\}$ and $f_- = -\min\{f,0\}$.
\end{lemma}
\begin{proof}
By representation of harmonic functions, recall $\lap f = 0$ in $B$, for any choice of $\delta > 0$ such that $K \subset \subset B_{1-\delta}$, we have 
\[
f (z) = \int_{B \backslash B_{1-\delta}} f(y)\, G(y,z)\, dy \quad \forall z \in K.
\]
Here $G$ can be computed explicitly, but the main point is that there exists $C = C(\delta,K) > 0$ so that
\[
\frac{1}{C(\delta,K)} \leq G(y,z) \leq C(\delta,K) \quad \mbox{for all $z \in K$, $y \in B \backslash B_{1-\delta}$.}
\]
In particular, for any fixed $z \in K$, 
\begin{equation}\label{eq:lambdahKest1}
\begin{split}
f(z) =& \int_{B \backslash B_{1-\delta}} f(y)\, G(y,z)\, dy\\
\aleq& -\int_{B \backslash B_{1-\delta}} (f)_-(y)\, G(y,z)\, dy  +  \int_{B} (f)_+.
\end{split}
\end{equation}
It remains to estimate $(f)_-$-term, 
\[
\begin{split}
&-\int_{B \backslash B_{1-\delta}} (f)_-(y)\, G(y,z)\, dy\\
\aleq & -\int_{B \backslash B_{1-\delta}} (f)_-(y)\, G(y,0)\, dy  \\
\leq & \int_{B \backslash B_{1-\delta}} f(y)\, G(y,0)\, dy \\
= & f(0)\\
\end{split}
\]
Now by the representation $f(0) = \mvint_{B(0,1)} f$, we continue
\[
\begin{split}
&-\int_{B \backslash B_{1-\delta}} (f)_-(y)\, G(y,z)\, dy\\
\aleq & \int_{B(0,1)} f\\
\leq & -\int_{B(0,1)} (f)_- + \int_{B} (f)_+\\
\end{split}
\]
Plugging this into \eqref{eq:lambdahKest1} we arrive at
\begin{equation}\label{eq:lambdahKest2}
\begin{split}
\sup_{z \in K} f(z) \leq& -C_2 \int_{B(0,1)} (f)_- + C_1 \int_{B} (f)_+
\end{split}
\end{equation}
\end{proof}

\subsection*{Acknowledgment}
A.S. was partially supported by the German Research Foundation (DFG) through grant no.~SCHI-1257-3-1, by the Daimler and Benz foundation through grant no. 32-11/16, as well as the Simons foundation through grant no 579261. A.S. was a Heisenberg Fellow.

A.S. likes to express his gratitude to E. Kuwert for many discussions on conformal parametrization and the results by M\"uller and Sverak.

\bibliographystyle{abbrv}%
\bibliography{bib}%

\end{document}